\documentclass[a4paper,11pt]{article}
\usepackage{amsfonts,amsmath,amssymb,amsthm,verbatim,placeins,xcolor,comment,placeins,caption,subcaption,enumerate,bm,graphicx}
\usepackage{geometry}
\usepackage[pdfauthor={DG},pdfstartview={FitH},colorlinks=true,citecolor=blue]{hyperref} 
\usepackage[english]{babel}
\usepackage[sort&compress]{natbib}
\usepackage{pgfplots}

\newtheorem{theorem}{Theorem}[section]

\newtheorem{lemma}{Lemma}[section]

\theoremstyle{remark}
\newtheorem{remark}{Remark}[section]
\theoremstyle{definition}

\theoremstyle{remark}
\newtheoremstyle{myremark}{}{}{\color{blue}\small}{}{\color{blue}\bfseries}{}{ }{}

\theoremstyle{myremark}

\newcommand{\E}{\mathbb{E}} 
\newcommand{\Var}{\operatorname{Var}} 
\newcommand{\R}{{\mathbb R}}

\newcommand\1{\mathbf{1}}

\DeclareMathOperator{\sign}{sign}
\newcommand{\mff}{\mathfrak{f}}

\begin{document}

\renewcommand*{\thefootnote}{\fnsymbol{footnote}}

\begin{center}
\Large{\textbf{The multifaceted behavior of integrated supOU processes:\\ The infinite variance case}}\\
\large{\today}\\
\bigskip
Danijel Grahovac$^1$\footnote{dgrahova@mathos.hr}, Nikolai N.~Leonenko$^2$\footnote{LeonenkoN@cardiff.ac.uk}, Murad S.~Taqqu$^3$\footnote{murad@bu.edu}\\
\end{center}

\bigskip
\begin{flushleft}
\footnotesize{
$^1$ Department of Mathematics, University of Osijek, Trg Ljudevita Gaja 6, 31000 Osijek, Croatia\\
$^2$ School of Mathematics, Cardiff University, Senghennydd Road, Cardiff, Wales, UK, CF24 4AG}\\
$^3$ Department of Mathematics and Statistics, Boston University, Boston, MA 02215, USA
\end{flushleft}

\bigskip

\textbf{Abstract: } The so-called ``supOU'' processes, namely the superpositions of Ornstein-Uhlenbeck type processes are stationary processes for which one can specify separately the marginal distribution and the dependence structure. They can have finite or infinite variance. We study the limit behavior of integrated infinite variance supOU processes adequately normalized. Depending on the specific circumstances, the limit can be fractional Brownian motion but it can also be a process with infinite variance, a L\'evy stable process with independent increments or a stable process with dependent increments. We show that it is even possible to have infinite variance integrated supOU processes converging to processes whose moments are all finite. A number of examples are provided. 

\bigskip


\bigskip


\bigskip

\section{Introduction}\label{sec1}
SupOU processes which are defined below are superpositions of stationary Ornstein-Uhlenbeck processes driven by a L\'evy process. They were studied extensively by Barndorff-Nielsen and his collaborators \cite{bn2001}, \cite{barndorff2011multivariate}, \cite{barndorff2013multivariate}, \cite{barndorff2013stochastic}. An attractive feature of supOU processes is that they allow the marginal distribution and the temporal dependence structure to be modeled independently. 

The \textit{supOU} process is  defined as follows: it is a strictly stationary process $X=\{X(t), \, t\in \R\}$ represented by the stochastic integral (\cite{bn2001})
\begin{equation}\label{supOU}
X(t)= \int_{\R_+} \int_{\R} e^{-\xi t + s} \mathbf{1}_{[0,\infty)}(\xi t -s) \Lambda(d\xi,ds).
\end{equation}
Here, $\Lambda$ is a homogeneous infinitely divisible random measure (\textit{L\'evy basis}) on $\R_+ \times \R$,  with cumulant function for $A \in \mathcal{B} \left(\R_+ \times \R\right)$
\begin{equation}\label{cumofLam}
C\left\{ \zeta \ddagger \Lambda(A)\right\} := \log \E e^{i \zeta \Lambda(A) } =  m(A) \kappa_{L}(\zeta) = \left( \pi \times Leb \right) (A) \kappa_{L}(\zeta).
\end{equation}
The \textit{control measure} $m=\pi \times Leb$ is the product of a probability measure $\pi$ on $\R_+$ and the Lebesgue measure on $\R$. The probability measure $\pi$ ``randomizes'' the rate parameter $\xi$ and the Lebesgue measure is associated with the moving average variable $s$. Finally, $\kappa_{L}$ in \eqref{cumofLam} is the cumulant function $\kappa_{L} (\zeta)= \log \E e^{i \zeta L(1) }$ of some infinitely divisible random variable $L(1)$ with L\'evy-Khintchine triplet $(a,b,\mu)$ i.e.
\begin{equation}\label{kappacumfun}
\kappa_{L}(\zeta) = i\zeta a -\frac{\zeta ^{2}}{2} b  +\int_{\R}\left( e^{i\zeta x}-1-i\zeta x \mathbf{1}_{[-1,1]}(x)\right) \mu(dx).
\end{equation}
The L\'evy process $L=\{L(t), \, t\geq 0\}$ associated with the triplet $(a,b,\mu)$ is called the \textit{background driving L\'{e}vy process} and the quadruple
\begin{equation}\label{quadruple}
(a,b,\mu,\pi)
\end{equation}
is referred to as the \textit{characteristic quadruple}. 

The marginal distribution of $X$ is determined by $L$, while the dependence structure is controlled by the probability measure $\pi$. Indeed, if $\E X(t)^2<\infty$, then the correlation function of $X$ is the Laplace transform of $\pi$:
\begin{equation}\label{corrsupOU}
r(t)=\int_{\R_+} e^{-t \xi }\pi (d\xi ), \quad t \geq 0.
\end{equation} 
More details about supOU processes can be found in \cite{bn2001}, \cite{barndorff2005spectral}, \cite{barndorff2013levy}, \cite{barndorff2011multivariate} and \cite{GLST2017Arxiv}.

Integrated supOU process $X^*=\{X^*(t), \, t \geq 0\}$ defined by
\begin{equation}\label{integratedsupOU}
X^*(t) = \int_0^t X(s) ds,
\end{equation}
has a complex asymptotic behavior. We have shown in \cite{GLT2017Limit} that when the supOU process has a finite variance, then different types of limits of integrated process can occur depending on the specific structure of the process. In this paper, we study what happens when the supOU has infinite variance. We show that again different limits can occur depending in particular on how heavy the tails of the supOU process are. We show that it is possible to have an infinite variance process to converge to a process with all moments finite.

Our results may be of particular interest in financial econometrics where supOU processes are used as stochastic volatility models and hence the integrated process $X^*$ represents the integrated volatility (see e.g.~\cite{barndorff2013multivariate}). The limiting behavior is also important for statistical estimation (see \cite{stelzer2015moment}, \cite{curato2019weak}). In \cite{GLT2017Limit} it has been shown that integrated supOU processes may exhibit an interesting phenomenon of intermittency which may be relevant for applications in turbulence (see e.g.~\cite{zel1987intermittency}).

When the supOU process $\{X(t), \, t\in \R\}$ has finite variance, four different limiting processes may be obtained depending on the elements of the characteristic quadruple, namely
\begin{itemize}
\item Brownian motion, 
\item fractional Brownian motion, 
\item a stable L\'evy process,
\item a stable process with dependent increments defined in \eqref{Zalphabeta} below.
\end{itemize} 
The type of limit depends on whether Gaussian component is present in \eqref{quadruple}, on a parameter $\alpha$ quantifying dependence and on a parameter $\beta$ quantifying the growth of the L\'evy measure $\mu$ in \eqref{quadruple} near origin.

We show in this paper that when the supOU process $\{X(t), \, t\in \R\}$ has infinite variance, the limiting behavior depends additionally on the regular variation index $\gamma$ of the marginal distribution. As limiting process, one can obtain
\begin{itemize}
\item a stable L\'evy process, 
\item a stable process with dependent increments defined in \eqref{Zalphabeta} below,
\item fractional Brownian motion. 
\end{itemize}
We provide examples to illustrate the results.

The paper is organized as follows. In Section \ref{sec2} we list the assumptions used for our results. Section \ref{sec3} contains the main results and in Section \ref{sec4} examples are provided. All the proofs are contained in Section \ref{sec:proofs}.

\section{Basic assumptions}\label{sec2}

Before stating the main results we introduce some notation and basic assumptions. 

\subsection{Preliminaries}

A random variable $Z$ with an infinite variance stable distribution $\mathcal{S}_\gamma (\sigma, \rho, c)$ and parameters $0<\gamma<2$, $\sigma>0$, $-1\leq \rho \leq 1$ and $c\in \R$ has a cumulant function of the form
\begin{equation}\label{cum:stable}
\kappa_{\mathcal{S}_\gamma (\sigma, \rho, c)}(\zeta) := C \{ \zeta \ddagger Z \} = i c \zeta - \sigma^{\gamma} |\zeta|^\gamma  \left( 1- i \rho \sign (\zeta) \chi (\zeta, \gamma) \right), \quad \zeta \in \R,
\end{equation}
where
\begin{equation*}
\chi (\zeta, \gamma) = \begin{cases}
\tan \left(\frac{\pi \gamma}{2}\right), & \gamma\neq 1,\\
\frac{\pi}{2} \log |\zeta|, & \gamma = 1.
\end{cases}
\end{equation*}
For simplicity of the exposition, wherever it applies we will assume $Z$ is symmetric ($\rho=0$) when $\gamma=1$, hence we can write
\begin{equation*}
\chi (\zeta, \gamma) = \chi(\gamma) = \begin{cases}
\tan \left(\frac{\pi \gamma}{2}\right), & \gamma\neq 1,\\
0, & \gamma = 1.
\end{cases}
\end{equation*}
We shall make a number of basic assumptions.

\subsection{Domain of attraction}
We suppose that the marginal distribution of the supOU process $\{X(t), \, t\in \R\}$ in \eqref{supOU} belongs to the domain of attraction of stable law, that is, $X(1)$ has balanced regularly varying tails:
\begin{equation}\label{regvarofX}
P(X(1)>x) \sim p k(x) x^{-\gamma} \quad \text{and} \quad P(X(1)\leq - x) \sim q k(x) x^{-\gamma}, \quad  \text{ as } x\to \infty,
\end{equation}
for some $p,q \geq 0$, $p+q>0$, $0<\gamma<2$ and some slowly varying function $k$. If $\gamma=1$, we assume $p=q$. In particular, the variance is infinite. Moreover, when the mean is finite, that is when $\gamma>1$, we assume $\E X(1)=0$. These assumptions imply that $X(1)$ is in the domain of attraction of $\mathcal{S}_\gamma (\sigma, \rho, 0)$ law with \cite[Theorem 2.6.1]{ibragimov1971independent}
\begin{equation}\label{sigmaandrho}
\sigma = \left( \frac{\Gamma(2-\gamma)}{1-\gamma} (p+q)  \cos \left(\frac{\pi \gamma}{2}\right) \right)^{1/\gamma}, \qquad \rho = \frac{p-q}{p+q}.
\end{equation}
Now consider the L\'evy process $\{L(t),\, t \geq 0\}$ introduced in Section \ref{sec1}. By \cite[Propositon 3.1]{fasen2007extremes}, the tail of the distribution function of $X(1)$ is asymptotically equivalent to the tail of the background driving L\'evy process $L(t)$ at $t=1$. More precisely, as $ x\to \infty$
\begin{equation}\label{equivalence of tails}
P(L(1)>x) \sim \gamma P(X(1)>x) \quad \text{and} \quad P(L(1)\leq - x) \sim \gamma P(X(1)\leq -x).
\end{equation}
Hence, \eqref{regvarofX} implies
\begin{equation}\label{regvarofL}
P(L(1)>x) \sim p \gamma k(x) x^{-\gamma} \quad \text{and} \quad P(L(1)\leq - x) \sim q \gamma k(x) x^{-\gamma}, \quad  \text{as } x\to \infty,
\end{equation}
and $L(1)$ is in the domain of attraction of stable distribution  $\mathcal{S}_\gamma (\gamma^{1/\gamma} \sigma, \rho, 0)$. Note that the scale parameter $\sigma$ of $X(1)$ yields a scale parameter $\gamma^{1/\gamma} \sigma$ for $L(1)$.

The normalizing sequence in some of the limit theorems below involves the de Bruijn conjugate of a slowly varying function \cite[Subsection 1.5.7]{bingham1989regular}. Recall that the de Bruijn conjugate of some slowly varying function $h$ is a slowly varying function $h^{\#}$ such that 
\begin{equation*}
h(x) h^{\#} \left(x h(x) \right) \to 1, \qquad h^{\#}(x) h (x h^{\#}(x)) \to 1,
\end{equation*}
as $x\to \infty$. By \cite[Theorem 1.5.13]{bingham1989regular} such function always exists and is unique up to asymptotic equivalence.

\subsection{Dependence structure}
The second set of assumptions deals with the temporal dependence structure dictated by the behavior near the origin of the probability measure $\pi$ in the characteristic quadruple \eqref{quadruple}. We will assume that the probability measure $\pi$ is regularly varying at zero, that is for some $\alpha>0$ and some slowly varying function $\ell$
\begin{equation}\label{regvarofpi}
\pi \left( (0,x] \right) \sim \ell(x^{-1}) x^{\alpha}, \quad \text{ as } x \to 0.
\end{equation}
To simplify the proofs of some of the results below, we will assume that $\pi$ has a density $p$ which is monotone on $(0,x')$ for some $x'>0$, so that \eqref{regvarofpi} implies
\begin{equation}\label{regvarofp}
p (x) \sim \alpha \ell(x^{-1}) x^{\alpha-1}, \quad \text{ as } x \to 0.
\end{equation}
To see how this affects dependence, note that if the variance is finite $\E X(t)^2<\infty$, then \eqref{corrsupOU} and \eqref{regvarofpi} imply that the correlation function satisfies \cite[Proposition 2.6]{fasen2007extremes}
\begin{equation*}
r( \tau) \sim \Gamma(1+\alpha) \ell(\tau) \tau^{-\alpha}, \qquad \text{ as } \tau \to \infty.
\end{equation*}
Hence, if $\alpha \in (0,1)$, the correlation function is not integrable, and the finite variance supOU process may be said to exhibit long-range dependence. On the other hand, note that the behavior of $\pi$ at infinity does not affect the decay of correlations as decay of correlations depends on the asymptotics of $\pi$ near zero. To simplify the presentation of the results, we shall assume that
\begin{equation}\label{pifinitemean}
\int_0^\infty \xi \pi(d\xi)<\infty.
\end{equation}

\subsection{Behavior of the L\'evy measure at the origin}
Unlike classical limit theorems, the limiting distribution of the integrated supOU processes does not depend only on the tails of the marginal distribution and on the dependence structure. The third component affecting the limit is the growth of the L\'evy measure $\mu$ near origin. We will quantify this growth by assuming a power law behavior of the L\'evy measure near the origin. Let
\begin{align*}
M^+(x) &= \mu \left( [x, \infty) \right), \quad x>0,\\
M^-(x) &= \mu \left( (-\infty, -x] \right), \quad x>0,
\end{align*}
denote the tails of $\mu$. We will assume that there exists $\beta\geq 0$, $c^+, c^- \geq 0$, $c^++c^->0$ such that
\begin{equation}\label{LevyMCond}
M^+(x) \sim c^+ x^{-\beta} \ \text{ and } \ M^-(x) \sim c^- x^{-\beta} \ \text{  as } x \to 0.
\end{equation}
Since $\mu$ is the L\'evy measure, we must have $\beta<2$. If \eqref{LevyMCond} holds, then $\beta$ is the Blumenthal-Getoor index of the L\'evy measure $\mu$ defined by (see \cite{GLT2017Limit})
\begin{equation}\label{BGindex}
\beta_{BG} = \inf \left\{\gamma \geq 0 : \int_{|x|\leq 1} |x|^\gamma \mu(dx) < \infty \right\}.
\end{equation}
Note that by \cite[Lemma 7.15]{kyprianou2014fluctuations} $M^+(x) \sim P(L(1)>x)$ and $M^-(x) \sim P(L(1)\leq -x)$ as $x \to \infty$, hence we can express \eqref{regvarofL} equivalently as
\begin{equation*}
M^+(x) \sim p \gamma k(x) x^{-\gamma} \quad \text{ and } \quad M^-(x) \sim q \gamma k(x) x^{-\gamma}, \quad  \text{ as } x\to \infty.
\end{equation*}
In general, making assumptions on the value of the Blumenthal-Getoor index $\beta_{BG}$ is more general than assuming \eqref{LevyMCond}. For example, in the geometric stable example in Subsection \ref{ex:geom:stable} below, the mass of the L\'evy measure near the origin increases at the logarithmic rate, hence \eqref{LevyMCond} does not hold but $\beta_{BG}=0$. Certain parts of our main results below require only assumptions on the  value of the Blumenthal-Getoor index and not \eqref{LevyMCond} (see Remark \ref{rem:con}).

The condition \eqref{LevyMCond} may be equivalently stated in terms of the L\'evy measure of $X(1)$. Indeed, if $\nu$ is the L\'evy measure of $X(1)$, then \eqref{LevyMCond} is equivalent to
\begin{equation}\label{LevyMCondforX}
\nu \left( [x, \infty) \right) \sim \beta^{-1} c^+ x^{-\beta} \ \text{ and } \ \nu \left( (-\infty, -x] \right) \sim \beta^{-1}  c^- x^{-\beta} \ \text{  as } x \to 0.
\end{equation}
See \cite{GLT2017Limit} for details.

\section{Main results}\label{sec3}

Before stating the main theorems, let us review the parameters introduced in the previous section:
\begin{itemize}
\item $\gamma \in (0,2)$ defined in \eqref{regvarofX} is the regular variation index of the marginal distribution,
\item $\alpha \in (0,\infty)$ defined in \eqref{regvarofp} quantifies the strength of dependence,
\item $\beta \in [0,2)$ defined in \eqref{LevyMCond} is the power law exponent of the L\'evy measure $\mu$ near origin.
\end{itemize}
The resulting limiting process depends on the interplay between the parameters $\alpha$, $\beta$ and $\gamma$. In the next theorem, the process $\{X(t), \, t\in \R\}$ has no Gaussian component. Here and in what follows, $\{\cdot\} \overset{fdd}{\to} \{\cdot\}$ denotes the convergence of finite dimensional distributions.

\begin{theorem}\label{thm:mainb=0}
Suppose that the supOU process $\{X(t), \, t\in \R\}$ is such that
\begin{itemize}
\item $b=0$, thus has no Gaussian component,
\item the marginal distribution satisfies \eqref{regvarofX} with $0<\gamma<2$,
\item the behavior at the origin of the L\'evy measure $\mu$ is given by \eqref{LevyMCond} with $0\leq \beta<2$,
\item $\pi$ has a density $p$ satisfying \eqref{regvarofp} with $\alpha>0$ and some slowly varying function $\ell$ and \eqref{pifinitemean} holds. 
\end{itemize}
Then the following holds:
\begin{enumerate}[(I)]
\item If $\gamma<1+\alpha$, then as $T\to \infty$
\begin{equation*}
\left\{ \frac{1}{T^{1/\gamma} k^{\#}(T)^{1/\gamma}} X^*(Tt) \right\} \overset{fdd}{\to} \left\{L_{\gamma} (t) \right\},
\end{equation*}
where $k$ is the slowly varying function in \eqref{regvarofX}, $k^{\#}$ is the de Bruijn conjugate of $1/k\left(x^{1/\gamma}\right)$ and the limit $\{L_{\gamma}\}$ is a $\gamma$-stable L\'evy process such that $L_{\gamma}(1)\overset{d}{=} \mathcal{S}_\gamma (\widetilde{\sigma}_{1,\gamma}, \rho, 0)$ with 
\begin{equation*}
\widetilde{\sigma}_{1,\gamma} = \sigma \left( \gamma \int_0^\infty \xi^{1-\gamma} \pi(d\xi) \right)^{1/\gamma},
\end{equation*}
and $\sigma$ and $\rho$ given by \eqref{sigmaandrho}.

\item If $\gamma>1+\alpha$, then the limit depends on the value of $\beta$, as follows.

\begin{enumerate}[({II}.a)]
\item If $\beta<1+\alpha$, then as $T\to \infty$
\begin{equation*}
\left\{ \frac{1}{T^{1/(1+\alpha)} \ell^{\#}\left(T \right)^{1/(1+\alpha)}} X^*(Tt) \right\} \overset{fdd}{\to} \left\{L_{1 + \alpha} (t) \right\},
\end{equation*}
where the limit $\{L_{1+\alpha}\}$ is a $(1+\alpha)$-stable L\'evy process such that $L_{1+\alpha}(1)\overset{d}{=} \mathcal{S}_{1+\alpha} (\widetilde{\sigma}, \widetilde{\rho}, 0)$ with 
\begin{equation*}
\widetilde{\sigma} = \left(\widetilde{\sigma}_{1,\beta}^{1+\alpha} + \widetilde{\sigma}_{2,\alpha}^{1+\alpha} \right)^{1/(1+\alpha)}, \qquad \widetilde{\rho} = \frac{\widetilde{\rho}_{1,\beta} \widetilde{\sigma}_{1,\beta}^{1+\alpha} + \widetilde{\rho}_{2,\alpha} \widetilde{\sigma}_{2,\alpha}^{1+\alpha}}{\widetilde{\sigma}_{1,\beta}^{1+\alpha} + \widetilde{\sigma}_{2,\alpha}^{1+\alpha}},
\end{equation*}
with $\widetilde{\sigma}_{1,\beta}$ and $\widetilde{\rho}_{1,\beta}$ defined in Lemma \ref{lemma:X1case2} and $\widetilde{\sigma}_{2,\alpha}$ and $\widetilde{\rho}_{2,\alpha}$ defined in Lemma \ref{lemma:X2LRD} below. 
\item If $1+\alpha<\beta$, then as $T\to \infty$
\begin{equation*}
\left\{ \frac{1}{T^{1-\alpha/\beta} \ell(T)^{1/\beta}} X^*(Tt) \right\} \overset{fdd}{\to} \left\{Z_{\alpha, \beta} (t) \right\},
\end{equation*}
where $\{Z_{\alpha, \beta}\}$ is a process with the stochastic integral representation
\begin{equation}\label{Zalphabeta}
Z_{\alpha, \beta}(t) = \int_{\R_+} \int_{\R} \left( \mathfrak{f}(\xi, t-s) - \mathfrak{f}(\xi, -s) \right) K(d\xi, ds),
\end{equation}
$\mathfrak{f}$ is given by
\begin{equation}\label{mff}
\mff(x,u) = \begin{cases}
1-e^{-xu}, & \text{ if } x>0 \text{ and } u>0,\\
0, & \text{ otherwise},
\end{cases}
\end{equation}
and $K$ is a $\beta$-stable L\'evy basis on $\R_+ \times \R$ with control measure $\alpha \xi^{\alpha} d\xi ds$ such that $C\left\{ \zeta \ddagger K(A)\right\} = \kappa_{\mathcal{S}_{\beta} (\widetilde{\sigma}_{2,\beta}, \widetilde{\rho}_{2,\beta}, 0)}(\zeta)$ with 
\begin{equation*}
\widetilde{\sigma}_{2,\beta} = \left( \frac{\Gamma(2-\beta)}{1-\beta} (c^-+c^+)  \cos \left(\frac{\pi \beta}{2}\right) \right)^{1/\beta}, \qquad \widetilde{\rho}_{2,\beta} = \frac{c^- - c^+}{c^-+c^+},
\end{equation*}
and  $c^-, c^+$ as in \eqref{LevyMCond}. The limit process $\{Z_{\alpha, \beta}\}$ has stationary increments and is self-similar with index $H=1-\alpha/\beta \in (1/\beta, 1)$.
\end{enumerate}
\end{enumerate}
\end{theorem}

\begin{remark}\label{rem:con}
We note that for the proof of Theorem \ref{thm:mainb=0}(I) when $\gamma<1$ one could replace \eqref{pifinitemean} with the assumption that there exists $\varepsilon>0$ such that
\begin{equation*}
\int_0^\infty \xi^{1-\gamma+\varepsilon} \pi(d\xi) < \infty.
\end{equation*}
Also, for the proof of Theorem \ref{thm:mainb=0}(II.a) instead of assuming \eqref{LevyMCond} with $\beta<1+\alpha$, it is enough to assume that the Blumenthal-Getoor index \eqref{BGindex} satisfies $\beta_{BG}<1+\alpha$. 
\end{remark}

The first boundary between different limit types in Theorem \ref{thm:mainb=0} is given by $\gamma=1+\alpha$. By choosing formally $\gamma=2$, we obtain $\alpha=1$ which corresponds to the boundary between short-range and long-range dependence in the finite variance case (see \cite{GLT2017Limit}). 

In the infinite variance case, the regular variation index $\gamma$ of the marginal tails seems to play an important role in the limit only when $\gamma<1+\alpha$. One could say that in this scenario \textit{the tails dominate the dependence structure}. In the opposite case $\gamma>1+\alpha$, two classes of stable processes may arise as a limit, either with dependent or independent increments. This depends on the value of parameter $\beta$. 

Note also that if $\beta<1+\alpha<\gamma$, the limiting process $L_{1+\alpha}$ has heavier tails than the supOU process whose tails are characterized by $\gamma$. On the other hand, when $1+\alpha<\gamma$ and $1+\alpha<\beta$ the limiting process has $\beta$-stable marginals hence, depending on whether $\beta>\gamma$ or $\beta<\gamma$, the tails of the limit can be lighter or heavier than the tails of the underlying supOU process.

We now consider the case when the Gaussian component is present in the characteristic quadruple, that is $b\neq 0$. This is the main difference between Theorem \ref{thm:mainb=0} and Theorem \ref{thm:mainb!=0}.

\begin{theorem}\label{thm:mainb!=0}
Suppose that the supOU process $\{X(t), \, t\in \R\}$ is such that
\begin{itemize}
\item $b\neq 0$, thus has a Gaussian component,
\item the marginal distribution satisfies \eqref{regvarofX} with $0<\gamma<2$,
\item the behavior at the origin of the L\'evy measure $\mu$ is given by \eqref{LevyMCond} with $0\leq \beta<2$,
\item $\pi$ has a density $p$ satisfying \eqref{regvarofp} with $\alpha>0$ and some slowly varying function $\ell$ and \eqref{pifinitemean} holds. 
\end{itemize}
\begin{enumerate}[(I)]
\item If $\alpha> 1$ or if $\alpha<1$ and $\gamma < \frac{2}{2-\alpha}$, then as $T\to \infty$
\begin{equation*}
\left\{ \frac{1}{T^{1/\gamma} k^{\#}(T)^{1/\gamma}} X^*(Tt) \right\} \overset{fdd}{\to} \left\{L_{\gamma} (t) \right\},
\end{equation*}
where the limit $\{L_{\gamma}\}$ is a $\gamma$-stable L\'evy process defined as in Theorem \ref{thm:mainb=0}(I). 

\item If $\alpha<1$ and $\gamma > \frac{2}{2-\alpha}$, then as $T\to \infty$
\begin{equation*}
\left\{ \frac{1}{T^{1-\alpha/2 } \ell(T)^{1/2}} X^*(Tt) \right\} \overset{fdd}{\to} \left\{\widetilde{\sigma}_{3,\alpha} B_H(t) \right\},
\end{equation*}
where $\{ B_H(t)\}$ is standard fractional Brownian motion with $H=1-\alpha/2$ and $\widetilde{\sigma}_{3,\alpha} = b^2/2  \frac{\Gamma(1+\alpha)}{(2-\alpha)(1-\alpha)}$.
\end{enumerate}
\end{theorem}

When the Gaussian component is present in the characteristic quadruple, the parameter $\beta$ is irrelevant for the type of the limit process and there are only two possible limits. One is the L\'evy stable motion $\{L_\gamma(t), \, t\geq 0\}$ that would have been a limit if $\{X^*(t), \, t\geq 0\}$ had independent increments. The second is the Gaussian fractional Brownian motion. In the first case, the limit has independent but infinite variance increments and in the second case the limit has dependent increments but their distribution is Gaussian.

Theorem \ref{thm:mainb!=0} also provides an example of a limit theorem where the aggregated process has infinite variance, but the limiting process is fractional Brownian motion which has all the moments finite. 

Figures \ref{fig:limits:Xb=01} and \ref{fig:limits:Xb!=01} illustrate the limiting behavior graphically.

\begin{figure}[h!]
\centering
\includegraphics[width=0.65\textwidth]{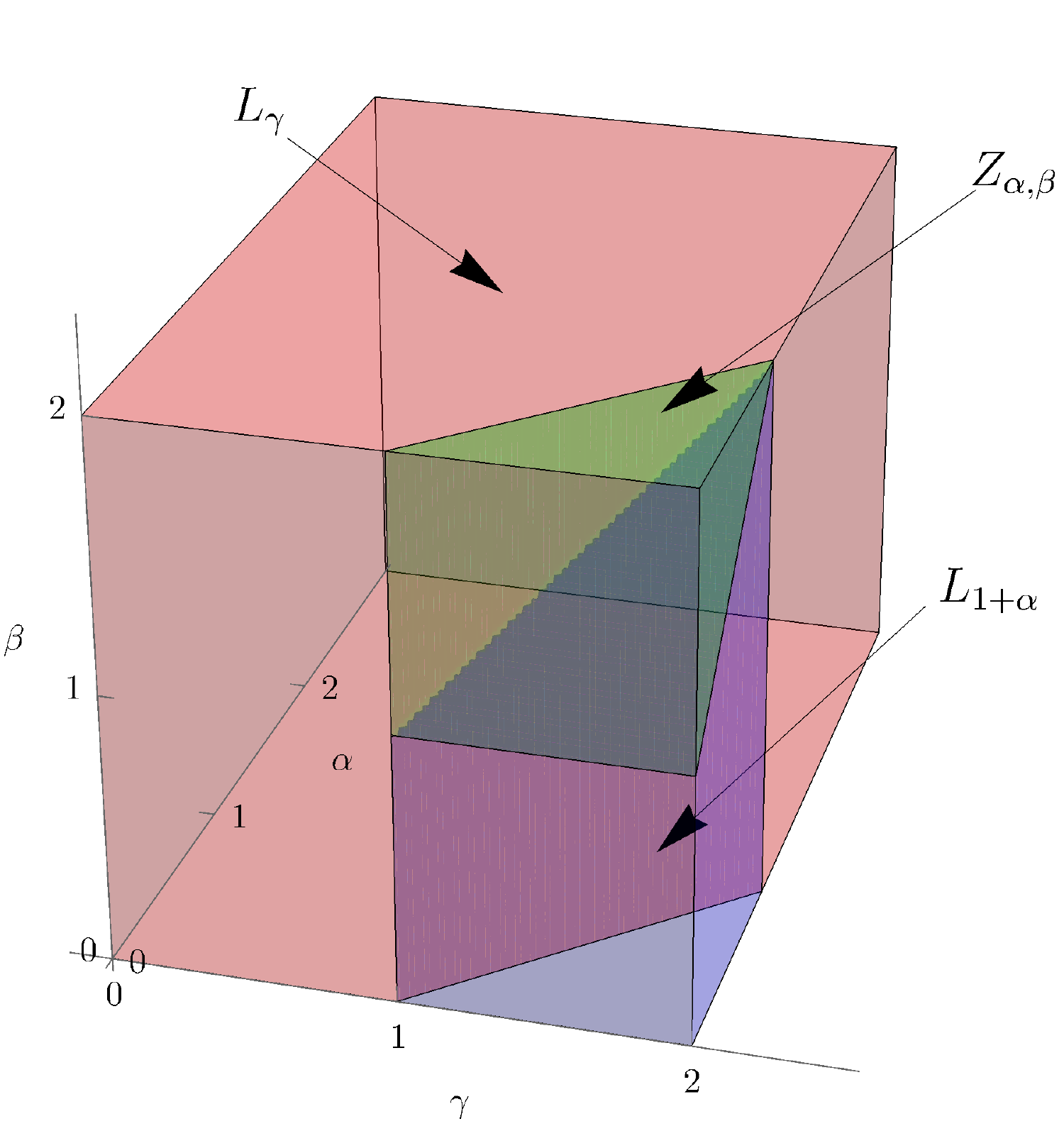}
\caption{Classification of limits of $X^*$ when $b=0$}
\label{fig:limits:Xb=01}
\end{figure}

\begin{figure}[h!]
\centering
\includegraphics[width=0.8\textwidth]{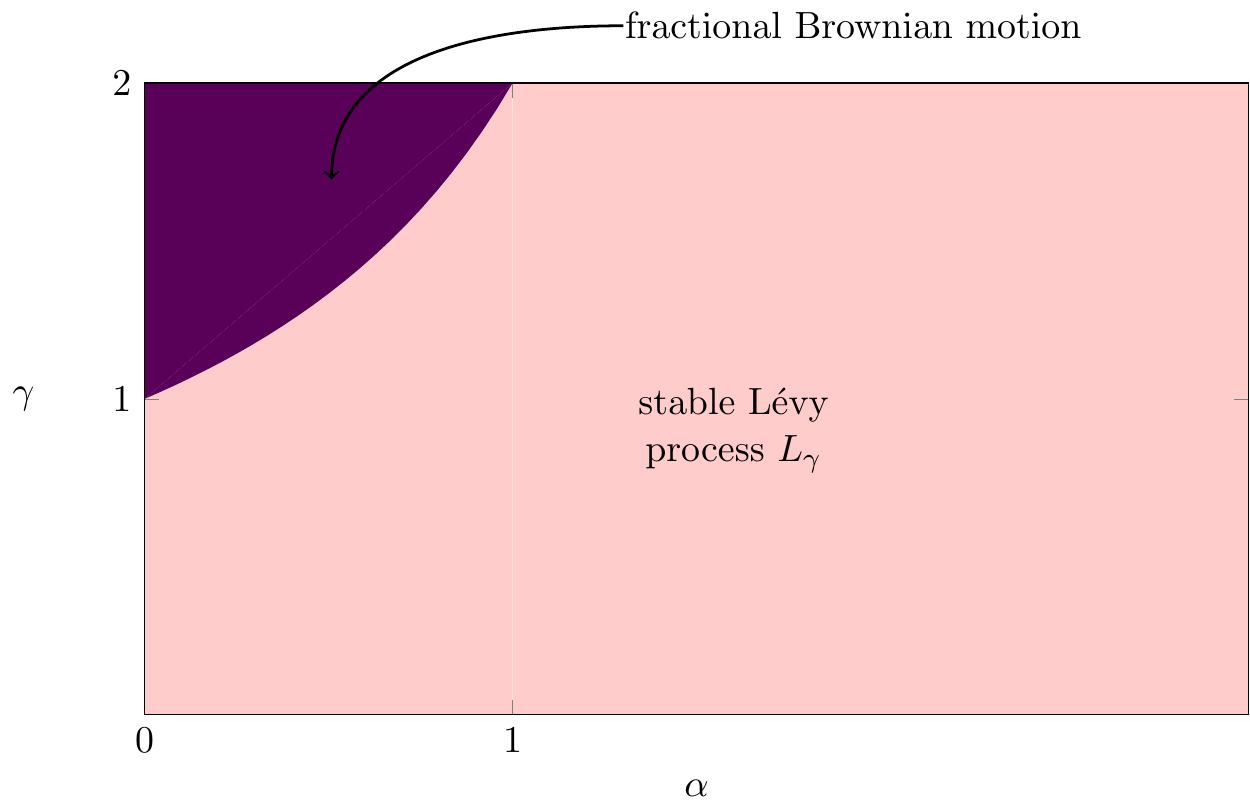}
\caption{Classification of limits of $X^*$ when $b\neq0$}
\label{fig:limits:Xb!=01}
\end{figure}

\section{Examples}\label{sec4}
In this section we list several examples of supOU process and show how Theorems \ref{thm:mainb=0} and \ref{thm:mainb!=0} apply. In each example we will fix the distribution of the background driving L\'evy process while $\pi$ may be any absolutely continuous probability measure satisfying \eqref{regvarofp}. For example, $\pi $ can be Gamma distribution with density
\begin{equation*}
  f(x)= \frac{1}{\Gamma(\alpha)} x^{\alpha-1} e^{-x} \mathbf{1}_{(0,\infty)}(x),
\end{equation*}
where $\alpha>0$. Then
\begin{equation*}
  \pi ((0,x]) \sim \frac{1}{\Gamma(\alpha+1)} x^{\alpha}, \quad \text{ as } x \to 0.
\end{equation*}
Other examples can be found in \cite{GLST2017Arxiv}.

In each of the examples bellow, we choose a background driving L\'evy process such that $L(1)$ is a heavy-tailed distribution satisfying \eqref{regvarofL} with $0<\gamma<2$ and \eqref{LevyMCond} holds or the Blumenthal-Getoor index \eqref{BGindex} is known. 

Note that by appropriately choosing the background driving L\'{e}vy process $L$, one can obtain any self-decomposable distribution as a marginal distribution of $X$. Recall that an infinitely divisible random variable $X$ is \textit{selfdecomposable} if its characteristic function $\phi (\theta)=\E e^{i\theta X}$, $\theta \in \R$, has the property that for every $c\in (0,1)$ there exists a characteristic function $\phi_{c}$ such that $\phi (\theta)=\phi (c\theta)\phi _{c}(\theta)$ for all $\theta\in \R$ (see e.g.~\cite{sato1999levy}). Equivalently, for every $c \in (0,1)$ there is a random variable $Y_c$ such that the random variable $X$ has the same distribution as $cX+Y_c$.

Each of distributions given in examples below may be imposed as a distribution of $X(t)$. Indeed, every distribution considered in the following examples is self-decompo\-sable (see references cited below), hence there exists a background driving L\'evy process generating a supOU process with such marginal distribution. Furthermore, if \eqref{regvarofX} holds, then $L(1)$  satisfies \eqref{regvarofL} by \eqref{equivalence of tails}. If \eqref{LevyMCondforX} holds for the L\'evy measure of $X(1)$, then this implies \eqref{LevyMCond} for the L\'evy measure of $L(1)$. Hence, Theorems \ref{thm:mainb=0} and \ref{thm:mainb!=0} may still be applied as the conditions on the background driving L\'evy process are easily translated to the corresponding conditions on the marginals of the supOU process.

\subsection{Compound Poisson background driving L\'evy process}
Let $L$ be a compound Poisson process with rate $\lambda>0$ and infinite variance jump distribution $F$ regularly varying at infinity. More precisely, $F$ satisfies
\begin{equation*}
F((x,\infty)) \sim p \gamma k(x) x^{-\gamma} \quad \text{ and } \quad F((-\infty,-x]) \sim q \gamma k(x) x^{-\gamma}, \quad  \text{ as } x\to \infty,
\end{equation*}
for some $0<\gamma<2$ and $k$ slowly varying at infinity. If $F$ has a finite mean, then we assume it is zero. Suppose $X$ is a supOU process with the background driving L\'evy process $L$ and $\pi$ absolutely continuous probability measure satisfying \eqref{regvarofp}. The characteristic quadruple \eqref{quadruple} is then $(a,0,\mu,\pi)$ where
\begin{equation*}
a=\lambda \int_{|x|\leq 1} x F(dx), \qquad \mu(dx) = \lambda F(dx).
\end{equation*}
Since the L\'evy measure is finite, this case corresponds to $\beta=0$ in \eqref{LevyMCond}. Hence, Theorem \ref{thm:mainb=0} applies to show that the limit is stable L\'evy process with index $\gamma$ if $\gamma<1+\alpha$ or with index $1+\alpha$ if $\gamma>1+\alpha$.

\subsection{Stable background driving L\'evy process}
Let $L$ be a $\gamma$-stable L\'evy process generating supOU process $X$ with characteristic quadruple \eqref{quadruple} given by $(a,0,\mu,\pi)$ where
\begin{equation*}
\mu(dx) = \begin{cases}
c_1 x^{-\gamma-1}dx,& x \in (0,\infty),\\
c_2 |x|^{-\gamma-1}dx,& x \in (-\infty,0),
\end{cases}
\end{equation*}
with $c_1,c_2\geq 0$, $c_1+c_2>0$ if $\gamma\neq 1$ and $c_1=c_2$ if $\gamma=1$. If $\alpha>1$, we additionally assume $\E X(1)=0$. The L\'evy measure satisfies \eqref{LevyMCond} with $\beta=\gamma$ and from Theorem \ref{thm:mainb=0} we conclude that if $\gamma<1+\alpha$, the limit is $\gamma$-stable L\'evy process and if $\gamma>1+\alpha$, then the limit is stable process $Z_{\alpha,\gamma}$ defined in Theorem \ref{thm:mainb=0} (II.b). This type of limiting behavior was obtained by \cite{puplinskaite2010aggregation} for aggregated AR(1) processes with stable marginals.

\subsection{Student's background driving L\'evy process}
Let $L$ be a L\'evy process such that $L(1)$ has Student's $t$-distribution given by the density
\begin{equation*}
f(x) =  \frac{\Gamma\left(\frac{\gamma+1}{2}\right)}{\delta \Gamma\left(\frac{1}{2}\right) \Gamma\left(\frac{\gamma}{2}\right)} 
\left(1+\left( \frac{x-c}{\delta} \right)^2 \right)^{-\frac{\gamma+1}{2}}, \quad x \in \mathbb{R},
\end{equation*}
where $c\in \R$ is location parameter, $\delta>0$ scale parameter and the degrees of freedom $0<\gamma<2$ correspond to the tail index of the distribution of $L(1)$ as in \eqref{regvarofL}. If $\gamma>1$, we assume $c=0$, hence $\E L(1)=0$. The L\'evy-Khintchine triplet in \eqref{kappacumfun} is $(c,0,\mu)$ with L\'evy measure $\mu$ absolutely continuous with density
\begin{equation*}
g(x) = \frac{1}{|x|} \int_0^{\infty} \frac{e^{-|x| \sqrt{2y}}}{\pi^2 y (J_{\gamma/2}^2(\delta \sqrt{2y}) + Y_{\gamma/2}^2(\delta \sqrt{2y})) } dy,
\end{equation*}
where $J_{\gamma/2}$ and $Y_{\gamma/2}$ denote the Bessel functions of the first and the second kind, respectively (see e.g.~\cite{heyde2005student}). By \cite[Eq.~(7.14)]{eberlein2004generalized} we have
\begin{equation*}
g(x)\sim \frac{\delta}{\pi} x^{-2}, \quad \text{ as } x\to 0,
\end{equation*}
and by using Karamata's theorem \cite[Theorem 1.5.11]{bingham1989regular} it follows that
\begin{equation*}
\mu \left( [x, \infty) \right) \sim \mu \left( (-\infty, -x] \right) \sim \frac{\delta}{\pi} x^{-1}, \quad \text{ as } x\to 0.
\end{equation*}
Hence, $\beta=1$ in \eqref{LevyMCond}. Let $\pi$ be an absolutely continuous probability measure satisfying \eqref{regvarofp}. Then the characteristic quadruple \eqref{quadruple} is $(c,0,\mu,\pi)$. By Theorem \ref{thm:mainb=0} the limits are as in the compound Poisson case, namely, stable L\'evy process with index $\gamma$ if $\gamma<1+\alpha$ or with index $1+\alpha$ if $\gamma>1+\alpha$.

\subsection{Geometric stable background driving L\'evy process}\label{ex:geom:stable}
A random variable $Y$ has a geometric stable distribution if its characteristic function has the form
\begin{equation*}
\E e^{i\zeta Y} = \frac{1}{1-\kappa_{\mathcal{S}_\gamma (\sigma, \rho, c)}(\zeta)}, \quad \zeta \in \R,
\end{equation*}
where $\kappa_{\mathcal{S}_\gamma (\sigma, \rho, c)}$ is the cumulant function \eqref{cum:stable} of some stable distribution $\mathcal{S}_\gamma (\sigma, \rho, c)$. The case $\rho=c=0$ yields the so-called Linnik distribution with characteristic function (\cite{kotz2001laplace, bakeerathan2008linnik})
\begin{equation*}
\E e^{i\zeta Y} = \frac{1}{1+\sigma^\gamma |\zeta|^\gamma}, \quad \zeta \in \R.
\end{equation*}
On the other hand, geometric stable distribution with $0<\gamma< 1$, $\sigma=\cos (\pi\gamma/2)^{1/\gamma}$, $\rho=1$ and $c=0$ is known as the Mittag-Leffler distribution (see \cite{kozubowski2001fractional}). 

Let $L$ be a L\'evy process such that $L(1)$ has geometric stable distribution. For $0<\gamma<2$, geometric stable distributions have regularly varying tails with index $\gamma$ (see e.g.~\cite{kozubowski1996moments}), hence \eqref{regvarofL} holds. On the other hand, the mass of the L\'evy measure near origin increases at the logarithmic rate, hence the Blumenthal-Getoor index \eqref{BGindex} is $0$ (see \cite{kozubowski1999tails} for details). Since the characteristic quadruple has no Gaussian component, we conclude from Theorem \ref{thm:mainb=0} and Remark \ref{rem:con} that the limit is stable L\'evy process with index $\gamma$ if $\gamma<1+\alpha$ or with index $1+\alpha$ if $\gamma>1+\alpha$.

\section{Proofs}\label{sec:proofs}

The proofs of Theorems \ref{thm:mainb=0} and \ref{thm:mainb!=0} are based on the L\'evy-It\^{o} decomposition of the background driving L\'evy process $L$ and the corresponding decomposition of the integrated process $X^*$. Let $\mu_1(dx)=\mu(dx) \1_{|x|>1}(dx)$ and $\mu_{2}(dx)=\mu(dx) \1_{|x|\leq1}(dx)$ where $\mu$ is the L\'evy measure of the L\'evy process $L$. Then there exists a modification of the L\'evy basis $\Lambda$ for which we can make a decomposition into $\Lambda_1$ with characteristic quadruple $(a,0,\mu_1,\pi)$, $\Lambda_2$ with characteristic quadruple $(0,0,\mu_2,\pi)$ and $\Lambda_3$ with characteristic quadruple $(0,b,0,\pi)$ (see \cite{pedersen2003levy}, \cite[Theorem 2.2]{barndorff2011multivariate} and \cite{moser2013functional}). We assume in the following $\Lambda$ is already a modification with L\'evy-It\^{o} decomposition. Let $L_1(t)$, $L_2(t)$ and $L_3(t)$, $t\in \R$ denote the corresponding background driving L\'evy processes which have the following cumulant functions:
\begin{align}
C \left\{ \zeta \ddagger L_1(1) \right\} & = i\zeta a +\int_{\R}\left( e^{i\zeta x}-1\right) \mu_1(dx) = i\zeta a +\int_{|x|>1}\left( e^{i\zeta x}-1\right) \mu(dx),\label{kappaL1}\\
C \left\{ \zeta \ddagger L_2(1) \right\} &=\int_{\R}\left( e^{i\zeta x}-1-i\zeta \mathbf{1}_{[-1,1]}(x)\right) \mu_2(dx)\nonumber\\
&=\int_{|x|\leq 1}\left( e^{i\zeta x}-1-i\zeta \mathbf{1}_{[-1,1]}(x)\right) \mu(dx) ,\nonumber\\
C \left\{ \zeta \ddagger L_3(1) \right\} &= -\frac{\zeta ^{2}}{2} b.\nonumber
\end{align}
Note that $L_1$ is a compound Poisson process and $L_3$ is Brownian motion. Consequently, we can represent $X(t)$ as
\begin{equation}\label{e:decomposition}
\begin{aligned}
X(t) &= \int_{0}^\infty \int_{-\infty }^{\xi t}e^{-\xi t + s} \Lambda_1(d\xi,ds) + \int_{0}^\infty \int_{-\infty }^{\xi t}e^{-\xi t + s} \Lambda_2(d\xi,ds)\\
&\hspace{4cm}+ \int_{0}^\infty \int_{-\infty }^{\xi t}e^{-\xi t + s} \Lambda_3(d\xi,ds)\\
&=: X_1(t) + X_2(t) + X_3(t),
\end{aligned}
\end{equation}
with $X_1$, $X_2$ and $X_3$ independent. Let $X^*_1$, $X^*_2$ and $X_3^*$ denote the corresponding integrated processes which are independent. To obtain the limiting behavior of the integrated process $X^*$ we first establish limit theorems for each process $X^*_1$, $X^*_2$ and $X_3^*$ separately.

\subsection{The process $X_1^*$}

When the supOU process has finite variance, then 
\begin{equation}\label{e:critcondFV}
\int_0^\infty \xi^{-1} \pi(d \xi)< \infty
\end{equation}
if and only if the correlation function is integrable (see \cite{GLT2017Limit}). If this is the case, then the integrated process after suitable normalization converges to Brownian motion. When the variance is infinite, then, assuming \eqref{regvarofX}, one may expect $\gamma$-stable L\'evy process in the limit.  

We first prove this for the compound Poisson component $X_1^*$. In this setting, the critical condition turns out to be
\begin{equation}\label{e:critcond}
\int_0^\infty \xi^{1-\gamma} \pi(d \xi)< \infty.
\end{equation}
Note that choosing formally $\gamma=2$ corresponds to the critical condition \eqref{e:critcondFV} in the finite variance case.

\begin{lemma}\label{lemma:X1case1}
Suppose that  there exists an $\varepsilon>0$ such that
\begin{equation}\label{thm:SRD1case:condition:gamale1}
\int_0^\infty \xi^{1-\gamma+\varepsilon} \pi(d\xi) < \infty \quad \text{ if } \gamma\in (0,1),
\end{equation}
or
\begin{equation}\label{thm:SRD1case:condition:gamagr1}
\int_0^\infty  \xi^{1-\gamma-\varepsilon} \pi(d\xi) < \infty  \quad \text{ if } \gamma\in [1,2).
\end{equation}
Then as $T\to \infty$
\begin{equation*}
\left\{ \frac{1}{T^{1/\gamma} k^{\#}(T)^{1/\gamma}} X_1^*(Tt) \right\} \overset{fdd}{\to} \left\{L_{\gamma} (t) \right\},
\end{equation*}
where the limit $\{L_{\gamma}\}$ is a $\gamma$-stable L\'evy process with the notation as in Theorem \ref{thm:mainb=0}(I).
\end{lemma}

\begin{proof}
Let $0=t_0<t_1<\cdots<t_m$, $\zeta_1,\dots,\zeta_m \in \R$ and $A_T=T^{1/\gamma} k^{\#}(T)^{1/\gamma}$. By the Cram{\'e}r-Wold device, it will be enough to prove that
\begin{equation*}
\sum_{i=1}^m \zeta_i A_T^{-1} X_1^*(Tt_{i}) \overset{d}{\to} \sum_{i=1}^m \zeta_i L_{\gamma} (t_i).
\end{equation*}
We can rewrite the left-hand side as
\begin{align*}
\sum_{i=1}^m \zeta_i \sum_{j=1}^i A_T^{-1}& \left( X_1^*(Tt_{j}) - X_1^*(Tt_{j-1}) \right)\\
&= \sum_{i=1}^m (m-i+1) \zeta_i A_T^{-1} \left( X_1^*(Tt_{i}) - X_1^*(Tt_{i-1}) \right)
\end{align*}
and the same can be done for the right-hand side. Hence, it is enough to prove that for arbitrary $\zeta_1,\dots,\zeta_m \in \R$
\begin{equation}\label{e:thm:SRD11}
\sum_{i=1}^m \zeta_i A_T^{-1} \left( X_1^*(Tt_{i}) - X_1^*(Tt_{i-1}) \right) \overset{d}{\to} \sum_{i=1}^m \zeta_i \left( L_{\gamma} (t_i) - L_{\gamma} (t_{i-1}) \right).
\end{equation}
By using \eqref{supOU} we have that
\begin{equation}\label{e:thmSRD1:decomposition}
\begin{aligned}
&X_1^*(Tt_{i}) - X_1^*(Tt_{i-1}) = \int_{T t_{i-1}}^{T t_i} \int_0^\infty \int_{-\infty}^{\xi u} e^{-\xi u + s} \Lambda_1(d\xi,ds) du  \\
&=\int_0^\infty \int_{-\infty}^{\xi T t_{i-1}} \int_{T t_{i-1}}^{T t_i} e^{-\xi u + s} du \Lambda_1(d\xi,ds) +  \int_0^\infty \int_{\xi T t_{i-1}}^{\xi T t_i} \int_{s/\xi}^{T t_i} e^{-\xi u + s} du \Lambda_1(d\xi,ds)\\
&=: \Delta X^*_{1,1}(Tt_{i}) + \Delta X^*_{1,2}(Tt_{i})
\end{aligned}
\end{equation}
with $\Delta X^*_{1,1}(Tt_{i})$ and $\Delta X^*_{1,2}(Tt_{i})$ independent. Moreover, $\Delta X^*_{1,2}(Tt_{i})$, $i=1,\dots,m$ are independent, hence, to prove \eqref{e:thm:SRD11}, it will be enough to prove that
\begin{align}
A_T^{-1} \Delta X^*_{1,1}(Tt_{i}) &\overset{d}{\to} 0,\label{e:thm:SRD1toprove1}\\
A_T^{-1} \Delta X^*_{1,2}(Tt_{i}) &\overset{d}{\to}  L_{\gamma} (t_i) - L_{\gamma} (t_{i-1}).\label{e:thm:SRD1toprove2}
\end{align}
Due to stationary increments, it is enough to consider $t_i=t_1=t$ so that $t_{i-1}=0$.\\

We start with the proof of \eqref{e:thm:SRD1toprove1}. For any $\Lambda$-integrable function $f$ on $\R_+ \times \R$, one has (see \cite{rajput1989spectral})
\begin{equation}\label{integrationrule}
C\left\{ \zeta \ddagger \int_{\R_+ \times \R}f d\Lambda \right\} = \int_{\R_+ \times \R} \kappa_{L} (\zeta f(\xi,s)) ds \pi(d\xi).
\end{equation}
Using this and the change of variables we get that
\begin{align}
C\left\{ \zeta \ddagger A_T^{-1} \Delta X^*_{1,1}(Tt) \right\} &= \int_0^\infty \int_{-\infty}^{0} \kappa_{L_1} \left( \zeta A_T^{-1} \int_{0}^{Tt} e^{-\xi u + s} du \right) ds \pi(d\xi)\nonumber \\
&= \int_0^\infty \int_{-\infty}^{0} \kappa_{L_1} \left( \zeta A_T^{-1} e^s \xi^{-1} \left(1-e^{-\xi Tt} \right)\right) ds \pi(d\xi).\label{eq:proof:SRD1X1}
\end{align}
By \cite[Theorem 2.6.4]{ibragimov1971independent}, the assumption \eqref{regvarofL} implies that 
\begin{equation}\label{eq:proof:SRD1X1:ibragimov}
\kappa_{L_1} (\zeta) \sim  k(1/|\zeta|) \kappa_{\mathcal{S}_\gamma (\gamma^{1/\gamma} \sigma, \rho, 0)}(\zeta), \quad \text{ as } \zeta \to 0.
\end{equation}
Hence, for arbitrary $\delta>0$, in some neighborhood of the origin one has
\begin{equation*}
|\kappa_{L_1}(\zeta) | \leq C_1 |\zeta|^{\gamma-\delta}, \quad |\zeta| \leq \varepsilon.
\end{equation*}
On the other hand, since $\left| e^{i\zeta x}-1 \right| \leq 2$, we have from \eqref{kappaL1} that
\begin{equation*}
|\kappa_{L_1}(\zeta) | \leq |a| |\zeta| +  2 \int_{\R} \mathbf{1}_{\{|x|> 1\}}(x) \mu(dx) \leq |a| |\zeta| + C_2.
\end{equation*}
We can take $C_3$ large enough so that $|\kappa_{L_1}(\zeta) | \leq C_3 |\zeta|$ for $|\zeta| > \varepsilon$ and then
\begin{equation}\label{eq:proof:SRD1X1:bound}
|\kappa_{L_1}(\zeta) | \leq C_1 |\zeta|^{\gamma-\delta} \1_{\{|\zeta| \leq \varepsilon\}} (\zeta) + C_3 |\zeta| \1_{\{|\zeta| > \varepsilon\}} (\zeta).
\end{equation}
Now we have by using \eqref{eq:proof:SRD1X1}
\begin{align*}
&\left| C\left\{ \zeta \ddagger A_T^{-1} \Delta X^*_{1,1}(Tt) \right\} \right|\\
&\leq C_1 \int_0^\infty \int_{-\infty}^{0} \left|\zeta A_T^{-1}  e^s \xi^{-1} \left(1-e^{-\xi Tt} \right)\right|^{\gamma-\delta} \1_{\{|\zeta A_T^{-1}  e^s \xi^{-1} \left(1-e^{-\xi Tt} \right)| \leq \varepsilon\}} (\zeta) ds \pi(d\xi)\\
&\quad + C_3 \int_0^\infty \int_{-\infty}^{0} \left|\zeta A_T^{-1}  e^s \xi^{-1} \left(1-e^{-\xi Tt} \right)\right| \1_{\{|\zeta A_T^{-1}  e^s \xi^{-1} \left(1-e^{-\xi Tt} \right)| > \varepsilon\}} (\zeta) ds \pi(d\xi)\\
&\leq C_1 |\zeta|^{\gamma-\delta} A_T^{-\gamma+\delta} \int_0^\infty \int_{-\infty}^{0} e^{(\gamma-\delta)s} \left(\xi^{-1} \left(1-e^{-\xi Tt} \right)\right)^{\gamma-\delta} ds \pi(d\xi)\\
& \quad + C_3 |\zeta| t A_T^{-1} T \int_0^\infty \int_{-\infty}^{0} e^{s} (\xi Tt)^{-1} \left(1-e^{-\xi Tt} \right) \1_{\{|\zeta A_T^{-1} \xi^{-1} \left(1-e^{-\xi Tt} \right)| > \varepsilon\}} (\zeta) ds \pi(d\xi)\\
& \leq C_1 \frac{1}{\gamma - \delta} |\zeta|^{\gamma-\delta} t^{\gamma-\delta} A_T^{-\gamma+\delta} T^{\gamma-\delta}\int_0^\infty \left((\xi Tt)^{-1} \left(1-e^{-\xi Tt} \right)\right)^{\gamma-\delta}  \pi(d\xi)\\
& \quad + C_3 |\zeta| t A_T^{-1} T \int_0^\infty (\xi Tt)^{-1} \left(1-e^{-\xi Tt} \right) \1_{\{|\zeta A_T^{-1} \xi^{-1} \left(1-e^{-\xi Tt} \right)| > \varepsilon\}} (\zeta)  \pi(d\xi)\\
& \leq C_1 \frac{1}{\gamma - \delta} |\zeta|^{\gamma-\delta} t^{\gamma-\delta} T^{\gamma-\delta-1+\delta/\gamma} k^{\#}(T)^{(-\gamma+\delta)/\gamma} \int_0^\infty \left((\xi Tt)^{-1} \left(1-e^{-\xi Tt} \right)\right)^{\gamma-\delta}  \pi(d\xi)\\
& \quad + C_3 |\zeta| t T^{1-1/\gamma} k^{\#}(T)^{-1/\gamma} \int_0^\infty (\xi Tt)^{-1} \left(1-e^{-\xi Tt} \right) \1_{\{|\zeta A_T^{-1} \xi^{-1} \left(1-e^{-\xi Tt} \right)| > \varepsilon\}} (\zeta)  \pi(d\xi).
\end{align*}

Now if $\gamma \in (0,1)$, then by using the inequality $x^{-1}(1-e^{-x})\leq 1$, $x>0$, and the fact that $\pi$ is a probability measure we have
\begin{align*}
&\left| C\left\{ \zeta \ddagger A_T^{-1} \Delta X^*_{1,1}(Tt) \right\} \right|\\
&\quad \leq  C_1 \frac{1}{\gamma - \delta} |\zeta|^{\gamma-\delta} t^{\gamma-\delta} T^{\gamma-\delta-1+\delta/\gamma} k^{\#}(T)^{(-\gamma+\delta)/\gamma} +  C_3 |\zeta| t T^{1-1/\gamma} k^{\#}(T)^{-1/\gamma}  \to 0,
\end{align*}
as $T\to \infty$, since $\gamma-\delta - 1 + \delta/\gamma < 0$ and $1-1/\gamma<0$. 

If $\gamma \in (1,2)$, then from the inequality $x^{-1}(1-e^{-x})\leq x^{-a}$ valid for $x>0$ and $a \in [0,1]$, we get by taking $a=a_1:=-(1-\gamma)/(\gamma- \delta)\in (0,1)$ for the first term and $a=a_2:=\gamma/2-1/(2\gamma)\in (0,1)$ for the second term that
\begin{align*}
&\left| C\left\{ \zeta \ddagger A_T^{-1} \Delta X^*_{1,1}(Tt) \right\} \right|\\
&\qquad \leq C_1 \frac{1}{\gamma - \delta} |\zeta|^{\gamma-\delta} t^{\gamma-\delta} T^{\gamma-\delta-1+\delta/\gamma} k^{\#}(T)^{(-\gamma+\delta)/\gamma} \int_0^\infty (\xi Tt)^{-a_1(\gamma-\delta)} \pi(d\xi)\\
&\qquad \quad + C_3 |\zeta| t T^{1-1/\gamma} k^{\#}(T)^{-1/\gamma} \int_0^\infty (\xi Tt)^{-a_2} \pi(d\xi)\\
&\qquad \leq C_1 \frac{1}{\gamma - \delta} |\zeta|^{\gamma-\delta} t^{1-\delta} T^{\delta/\gamma -\delta} k^{\#}(T)^{(-\gamma+\delta)/\gamma} \int_0^\infty \xi^{1-\gamma} \pi(d\xi)\\
&\qquad \quad +  C_3 |\zeta| t^{1-a_{2}} T^{1-1/\gamma-a_2} k^{\#}(T)^{-1/\gamma}  \int_0^\infty \xi^{-a_2} \pi(d\xi).
\end{align*}
This tends to zero as $T\to \infty$ since $\delta/\gamma -\gamma<0$, $1-1/\gamma-a_2<0$ and $\int_0^\infty \xi^{-a_2} \pi(d\xi)<\infty$ due to $-a_2>1-\gamma$. 

If $\gamma=1$, then we can similarly take $a=a_1=\varepsilon/(\gamma-\delta) \in (0,1)$ for the first term and $a=a_2:=\varepsilon \in (0,1)$ for the second term to obtain
\begin{align*}
&\left| C\left\{ \zeta \ddagger A_T^{-1} \Delta X^*_{1,1}(Tt) \right\} \right|\\
&\qquad \leq C_1 \frac{1}{\gamma - \delta} |\zeta|^{\gamma-\delta} t^{1-\delta-\varepsilon} T^{-\varepsilon} k^{\#}(T)^{(-\gamma+\delta)/\gamma} \int_0^\infty \xi^{-\varepsilon} \pi(d\xi)\\
&\qquad \quad +  C_3 \frac{1}{2} |\zeta| t^{1-\varepsilon} T^{-\varepsilon} k^{\#}(T)^{-1/\gamma}  \int_0^\infty \xi^{-\varepsilon} \pi(d\xi) \to 0, \quad \text{ as } T\to \infty.
\end{align*}
This completes the proof of \eqref{e:thm:SRD1toprove1}.\\

To prove \eqref{e:thm:SRD1toprove2}, note that because of \eqref{eq:proof:SRD1X1:ibragimov} we can write
\begin{equation*}
\kappa_{L_1} (\zeta) = \overline{k}(\zeta) \kappa_{\mathcal{S}_\gamma (\gamma^{1/\gamma} \sigma, \rho, 0)}(\zeta),
\end{equation*}
with $\overline{k}$ slowly varying at zero such that $\overline{k}(\zeta)\sim k(1/\zeta)$ as $\zeta \to 0$. From \eqref{integrationrule} we have
\begin{align}
&C \left\{ \zeta \ddagger A_T^{-1} \Delta X^*_{1,2}(Tt) \right\}\nonumber \\
&= \int_0^\infty \int_{0}^{\xi T t} \kappa_{L_1} \left( \zeta A_T^{-1} \int_{s/\xi}^{Tt} e^{-\xi u + s} du \right) ds \pi(d\xi)\nonumber \\
&= \int_0^\infty \int_{0}^{\xi T t} \kappa_{L_1} \left( \zeta A_T^{-1} \xi^{-1} \left( 1 - e^{-\xi Tt+s} \right) \right) ds \pi(d\xi)\nonumber \\
&= \int_0^\infty \int_{0}^{t} \kappa_{L_1} \left( \zeta A_T^{-1} \xi^{-1} \left( 1 - e^{-\xi T (t-s)} \right) \right) \xi T ds \pi(d\xi)\nonumber \\
&= \int_0^\infty \int_{0}^{t} \overline{k}\left( \zeta A_T^{-1} \xi^{-1} \left( 1 - e^{-\xi T (t-s)} \right) \right)\nonumber\\
&\hspace{0.5cm} \times \kappa_{\mathcal{S}_\gamma (\gamma^{1/\gamma} \sigma, \rho, 0)}\left( \zeta A_T^{-1} \xi^{-1} \left( 1 - e^{-\xi T (t-s)} \right) \right) \xi T ds \pi(d\xi)\nonumber \\
&= \kappa_{\mathcal{S}_\gamma (\gamma^{1/\gamma} \sigma, \rho, 0)}\left( \zeta \right) \int_0^\infty \int_{0}^{t}  A_T^{-\gamma} \left(\xi^{-1} \left( 1 - e^{-\xi T (t-s)} \right) \right)^{\gamma}\nonumber\\
&\hspace{3cm} \times \overline{k}\left( \zeta A_T^{-1} \xi^{-1} \left( 1 - e^{-\xi T (t-s)} \right) \right)  \xi T ds \pi(d\xi)\nonumber \\
&= \kappa_{\mathcal{S}_\gamma (\gamma^{1/\gamma} \sigma, \rho, 0)}\left( \zeta \right)\nonumber \int_0^\infty \int_{0}^{t}  \xi^{1-\gamma} \left( 1 - e^{-\xi T (t-s)} \right)^{\gamma} \nonumber\\
&\hspace{3cm} \times \frac{\overline{k}\left( \left(T k^{\#}(T) \right)^{-1/\gamma} \zeta \xi^{-1} \left( 1 - e^{-\xi T (t-s)} \right) \right)}{k^{\#}(T)}  ds \pi(d\xi).\label{e:thmSRD1:proof2lastline}
\end{align}
By the definition of $k^{\#}$, one has \cite[Theorem 1.5.13]{bingham1989regular}
\begin{equation*}
\frac{k^{\#} (T)}{ \overline{k} \left( \left( T k^{\#} (T) \right)^{-1/\gamma} \right)} \sim \frac{k^{\#} (T)}{ k \left( \left( T k^{\#} (T) \right)^{1/\gamma} \right)}  \to 1, \quad \text{ as } T\to \infty,
\end{equation*}
and due to slow variation of $\overline{k}$, for any $\zeta\in \R$, $\xi>0$ and $s\in (0,t)$, as $T\to \infty$
\begin{equation}\label{e:thmSRD1:proof1}
\begin{aligned}
&\frac{k^{\#} (T)}{\overline{k}\left( \left(T k^{\#}(T) \right)^{-1/\gamma} \zeta \xi^{-1} \left( 1 - e^{-\xi T (t-s)} \right) \right)} =\\
&\qquad \frac{\overline{k} \left( \left( T k^{\#} (T) \right)^{-1/\gamma} \right) }{\overline{k}\left( \left(T k^{\#}(T) \right)^{-1/\gamma} \zeta \xi^{-1} \left( 1 - e^{-\xi T (t-s)} \right) \right)}  \frac{k^{\#} (T)}{ \overline{k} \left( \left( T k^{\#} (T) \right)^{-1/\gamma} \right)} \to 1.
\end{aligned}
\end{equation}
Hence, if the limit could be passed under the integral in \eqref{e:thmSRD1:proof2lastline}, we would get that
\begin{equation*}
C \left\{ \zeta \ddagger A_T^{-1} \Delta X^*_{1,2}(Tt) \right\} \to t \kappa_{\mathcal{S}_\gamma (\gamma^{1/\gamma} \sigma, \rho, 0)} \int_0^\infty \xi^{1-\gamma} \pi(d\xi), \quad \text{ as } T\to \infty,
\end{equation*}
which proves \eqref{e:thm:SRD1toprove2}. To justify taking the limit under the integral, note that by Potter's bounds \cite[Theorem 1.5.6]{bingham1989regular} we have from \eqref{e:thmSRD1:proof1} that for any $\delta>0$ 
\begin{align*}
&\frac{\overline{k}\left( \left(T k^{\#}(T) \right)^{-1/\gamma} \zeta \xi^{-1} \left( 1 - e^{-\xi T (t-s)} \right) \right)}{k^{\#}(T)}\\
&\qquad \leq C_5 \max \left\{ \zeta^\delta \xi^{-\delta} \left( 1 - e^{-\xi T (t-s)} \right)^\delta , \zeta^{-\delta} \xi^{\delta} \left( 1 - e^{-\xi T (t-s)} \right)^{-\delta}\right\}\\
&\qquad \leq C_6 \left( 1 - e^{-\xi T (t-s)} \right)^{-\delta} \max \left\{ \xi^{-\delta}, \xi^\delta \right\},
\end{align*}
for $T$ large enough. By taking $\delta<\min \{\gamma, \varepsilon\}$ we get
\begin{align*}
\xi^{1-\gamma} \left( 1 - e^{-\xi T (t-s)} \right)^{\gamma} &\frac{\overline{k}\left( \left(T k^{\#}(T) \right)^{-1/\gamma} \zeta \xi^{-1} \left( 1 - e^{-\xi T (t-s)} \right) \right)}{k^{\#}(T)}\\
&\leq  C_6 \xi^{1-\gamma} \left( 1 - e^{-\xi T (t-s)} \right)^{\gamma-\delta} \max \left\{ \xi^{-\delta}, \xi^\delta \right\}\\
&\leq  C_6 \xi^{1-\gamma} \max \left\{ \xi^{-\delta}, \xi^\delta \right\}
\end{align*}
and by the assumptions \eqref{thm:SRD1case:condition:gamale1} and \eqref{thm:SRD1case:condition:gamagr1}
\begin{align*}
&\int_0^\infty \int_{0}^{t}  \xi^{1-\gamma}  \max \left\{ \xi^{-\delta}, \xi^\delta \right\} ds \pi(d\xi)\\
&\hspace{4cm}= t \int_0^1  \xi^{1-\gamma-\delta} \pi(d\xi) + t\int_1^\infty \xi^{1-\gamma+\delta} \pi(d\xi) < \infty.
\end{align*}
Hence, the dominated convergence theorem can be applied in \eqref{e:thmSRD1:proof2lastline}.
\end{proof}

We next consider a scenario where \eqref{regvarofp} holds. If $\gamma \in (1,2)$, then this implies that \eqref{e:critcond} does not hold.

\begin{lemma}\label{lemma:X1case2}
Suppose that $\pi$ has a density $p$ satisfying \eqref{regvarofp} with $\alpha\in(0,1)$ and some slowly varying function $\ell$. If 
\begin{equation*}
1 + \alpha < \gamma,
\end{equation*}
then as $T\to \infty$
\begin{equation}\label{e:X1*limitLRD}
\left\{ \frac{1}{T^{1/(1+\alpha)} \ell^{\#}\left(T \right)^{1/(1+\alpha)}} X_1^*(Tt) \right\} \overset{fdd}{\to} \left\{L_{1 + \alpha} (t) \right\},
\end{equation}
where $\ell^{\#}$ is de Bruijn conjugate of $1/\ell\left(x^{1/(1+\alpha)}\right)$ and the limit $\{L_{1+\alpha}\}$ is $(1+\alpha)$-stable L\'evy process such that $L_{1+\alpha}(1)\overset{d}{=} \mathcal{S}_\gamma (\widetilde{\sigma}_{1,\alpha}, \widetilde{\rho}_1, 0)$ with
\begin{equation}\label{sigma1alpha}
\widetilde{\sigma}_{1,\alpha} = \left( \frac{\Gamma(1-\alpha)}{\alpha} (c^-_1+c^+_1)  \cos \left(\frac{\pi (1+\alpha)}{2}\right) \right)^{1/(1+\alpha)}, \qquad \widetilde{\rho_1} = \frac{c^-_1 - c^+_1}{c^-_1+c^+_1},
\end{equation}
and  $c^-_1, c^+_1$ given by
\begin{equation}\label{c-+1}
c^-_1 = \frac{\alpha}{1+\alpha} \int_{-\infty}^{-1} |y|^{1+\alpha} \mu(dy), \qquad c^+_1 = \frac{\alpha}{1+\alpha} \int_1^{\infty} y^{1+\alpha} \mu(dy).
\end{equation}
\end{lemma}

\begin{proof}
The proof is similar to the proof of \cite[Theorem 2.2]{GLT2017Limit}. As in the proof of Lemma \ref{lemma:X1case1}, it will be enough to prove that as $T\to \infty$
\begin{align}
A_T^{-1} \Delta X^*_{1,1}(Tt) &\overset{d}{\to} 0,\label{e:thm:LRD1toprove1}\\
A_T^{-1} \Delta X^*_{1,2}(Tt) &\overset{d}{\to}  L_{1+\alpha} (t),\label{e:thm:LRD1toprove2}
\end{align}
with $A_T=T^{1/(1+\alpha)} \ell^{\#}\left(T \right)^{1/(1+\alpha)}$.  Note that the de Bruijn conjugate $\ell^{\#}$ exists by \cite[Theorem 1.5.13]{bingham1989regular} and satisfies
\begin{equation}\label{e:thmSPL2dBcon}
\frac{\ell^{\#} \left(T\right)}{\ell\left( \left( T\ell^{\#}\left(T \right) \right)^{1/(1+\alpha)} \right) } \sim 1, \text{  as } T\to \infty.
\end{equation}

To prove \eqref{e:thm:LRD1toprove1}, note that we can write $p(x)=\alpha \widetilde{\ell}(x^{-1}) x^{\alpha-1}$ where $\widetilde{\ell}(t) \sim \ell(t)$ as $t \to \infty$. Now from \eqref{eq:proof:SRD1X1} we have
\begin{align*}
C&\left\{ \zeta \ddagger A_T^{-1} \Delta X^*_{1,1}(Tt) \right\}\\
&= \int_0^\infty \int_{-\infty}^{0} \kappa_{L_1} \left( \zeta A_T^{-1} T e^s \xi^{-1} \left(1-e^{-\xi t} \right)\right) ds \pi(T^{-1} d\xi)\\
&= \int_0^\infty \int_{-\infty}^{0} \kappa_{L_1} \left( \zeta A_T^{-1} T e^s \xi^{-1} \left(1-e^{-\xi t} \right)\right) ds \pi(T^{-1} d\xi)\\
&= \int_0^\infty \int_{-\infty}^{0} \kappa_{L_1} \left( \zeta A_T^{-1} T e^s \xi^{-1} \left(1-e^{-\xi t} \right)\right)  \alpha \widetilde{\ell}(T \xi^{-1}) \xi^{\alpha-1} T^{-\alpha} ds d\xi.
\end{align*}
We have assumed $1+\alpha<\gamma$, hence $\gamma>1$ and from \eqref{eq:proof:SRD1X1:bound} we get the bound
\begin{equation}\label{e:LRDcaseboundlinear}
|\kappa_{L_1}(\zeta)| \leq C_1 |\zeta|,\quad \zeta \in \R.
\end{equation}
By using Potter's bounds \cite[Theorem 1.5.6]{bingham1989regular} we have for $0<\delta<\alpha^2/(1+\alpha)$
\begin{equation*}
\widetilde{\ell}(T \xi^{-1}) = \frac{\widetilde{\ell}(T \xi^{-1})}{\widetilde{\ell}( \xi^{-1})} \widetilde{\ell}(\xi^{-1}) \leq C_2 \max \left\{ T^{-\delta}, T^{\delta} \right\} \widetilde{\ell}(\xi^{-1}).
\end{equation*}
Now we get that
\begin{align*}
&\left| C\left\{ \zeta \ddagger A_T^{-1} \Delta X^*_{1,1}(Tt) \right\}  \right|\\
&\ \leq \alpha C_3 |\zeta| T^{-\alpha^2/(1+\alpha)+\delta}  \ell^{\#}\left(T \right)^{-1/(1+\alpha)} \int_0^\infty \int_{-\infty}^{0} e^{s} \xi^{-1} \left(1-e^{-\xi t} \right) \widetilde{\ell}(\xi^{-1}) \xi^{\alpha-1} ds d\xi \\
&\ \leq C_4  |\zeta| T^{-\alpha^2/(1+\alpha)+\delta}  \ell^{\#}\left(T \right)^{-1/(1+\alpha)} \int_0^\infty \widetilde{\ell}(\xi^{-1}) \xi^{\alpha-1} d\xi \to 0,
\end{align*}
as $T\to \infty$.\\

We now turn to \eqref{e:thm:LRD1toprove2}. As in the proof of Lemma \ref{lemma:X1case1} we have
\begin{align}
C &\left\{ \zeta \ddagger A_T^{-1} \Delta X^*_{1,2}(Tt) \right\}\nonumber\\
&= \int_0^\infty \int_{0}^{t} \kappa_{L_1} \left( \zeta A_T^{-1} \xi^{-1} \left( 1 - e^{-\xi T (t-s)} \right) \right) \xi T ds \pi(d\xi)\nonumber\\
&= \int_0^\infty \int_{0}^{t} \kappa_{L_1} \left( \zeta A_T^{-1} \xi^{-1} \left( 1 - e^{-\xi T (t-s)} \right) \right) \alpha \widetilde{\ell}(\xi^{-1}) \xi^{\alpha} T ds d\xi. \label{e:thmSPL22}
\end{align}
Suppose that $\zeta>0$. The proof is analogous if $\zeta<0$. Making change of variables $x=\zeta A_T^{-1} \xi^{-1}$ in \eqref{e:thmSPL22} we get
\begin{align}
&C \left\{ \zeta \ddagger A_T^{-1} \Delta X^*_{1,2}(Tt) \right\}\nonumber\\
&= \zeta^{1+\alpha} \int_0^\infty \int_{0}^{t} \kappa_{L_1} \left( x \left( 1 - e^{-x^{-1} \frac{\zeta T}{A_T} (t-s)} \right) \right)  A_T^{-(1+\alpha)} T \widetilde{\ell}\left(A_T x \zeta^{-1}\right) \alpha x^{-\alpha-2} ds dx\nonumber\\
&= \zeta^{1+\alpha} \int_0^\infty \int_{0}^{t} \kappa_{L_1} \left( x \left( 1 - e^{-x^{-1} \frac{\zeta T}{A_T} (t-s)} \right) \right)\nonumber\\
&\hspace{4cm} \times \frac{\widetilde{\ell}\left( T^{1/(1+\alpha)} \ell^{\#}\left(T \right)^{1/(1+\alpha)} x \zeta^{-1} \right) }{\ell^{\#} \left(T\right)} \alpha x^{-\alpha-2} ds dx, \label{e:44}
\end{align}
and $T/A_T\to \infty$ as $T\to \infty$ implies that
\begin{equation*}
\kappa_{L_1} \left( x \left( 1 - e^{-x^{-1} \frac{\zeta T}{A_T} (t-s)} \right) \right) \to \kappa_{L_1} ( x ).
\end{equation*}
Since $\ell$ is slowly varying, $\ell\sim \widetilde{\ell}$ and \eqref{e:thmSPL2dBcon} holds, we have
\begin{align*}
&\frac{\widetilde{\ell}\left( T^{1/(1+\alpha)} \ell^{\#}\left(T \right)^{1/(1+\alpha)} x \zeta^{-1} \right) }{\ell^{\#} \left(T\right)} \frac{\ell\left( T^{1/(1+\alpha)} \ell^{\#}\left(T \right)^{1/(1+\alpha)}\right) }{\ell\left( T^{1/(1+\alpha)} \ell^{\#}\left(T \right)^{1/(1+\alpha)}\right)}\\
& \hspace{5em} \sim \frac{\ell\left( \left( T\ell^{\#}\left(T \right) \right)^{1/(1+\alpha)} \right) }{\ell^{\#} \left(T\right)} \to 1,
\end{align*}
as $T \to \infty$. Hence, if the limit could be passed under the integral, we would get that
\begin{equation}\label{e:thm:SLP3}
C \left\{ \zeta \ddagger A_T^{-1} \Delta X^*_{1,2}(Tt) \right\} \to t \zeta^{1+\alpha} \int_0^\infty \kappa_{L_1} (x) \alpha x^{-\alpha-2} dx.
\end{equation}
Let us assume momentarily that \eqref{e:thm:SLP3} holds. Since $\gamma>1$, we have assumed that the mean is $0$, namely $\E X_1(1)=\E L_1(1)= a + \int_{|x|>1} x \mu(dx)=0$  and hence from \eqref{kappaL1} we can write $\kappa_{L_1}$  in the form
\begin{equation}\label{e:kappaL1alternativeform}
\kappa_{L_1} (\zeta) = \int_{|x|>1}\left( e^{i\zeta x}-1 - i\zeta x\right) \mu(dx) = \int_{\R}\left( e^{i\zeta x}-1 - i\zeta x\right) \mu_1(dx).
\end{equation}
By using the relation 
\begin{equation*}
\int_0^\infty \left(e^{ \mp iu} -1\pm iu\right)  u^{-\lambda-1} du = \exp \left\{  \mp \frac{1}{2} i \pi \lambda \right\} \frac{\Gamma(2-\lambda)}{\lambda (\lambda-1)}
\end{equation*}
valid for $1<\lambda<2$ (see e.g.~\cite[Theorem 2.2.2]{ibragimov1971independent}), we obtain by taking $\lambda=1+\alpha$ that
\begin{align*}
&\alpha \int_0^\infty \kappa_{L_1} (x)  x^{-\alpha-2} dx = \alpha \int_{-\infty}^\infty \int_0^{\infty} \left(e^{ixy} -1-ixy\right)  x^{-\alpha-2} dx \mu_1(dy)\\
&= \alpha \int_{0}^\infty \int_0^{\infty} \left(e^{iu} -1-iu\right)  u^{-\alpha-2} du y^{1+\alpha}\mu_1(dy)\\
&\hspace{6em}  + \alpha \int_{-\infty}^0 \int_0^{\infty} \left(e^{-iu} -1+iu\right)  u^{-\alpha-2} du (-y)^{1+\alpha}\mu_1(dy)\\
&= \frac{\alpha\Gamma (1-\alpha)}{(1+\alpha)\alpha}  \left( e^{i(1+\alpha) \pi /2} \int_{0}^\infty y^{1+\alpha}\mu_1(dy) +  e^{-i(1+\alpha) \pi /2} \int_{-\infty}^0 |y|^{1+\alpha}\mu_1(dy) \right)\\
&= \frac{\Gamma (1-\alpha)}{\alpha} \Bigg( \cos \left(\frac{\pi (1+\alpha)}{2}\right) \left( \frac{\alpha}{1+\alpha} \int_{-\infty}^{-1} |y|^{1+\alpha} \mu(dy)  +  \frac{\alpha}{1+\alpha} \int_{1}^\infty y^{1+\alpha} \mu(dy)\right)\\
&\hspace{3em} -  i \sin \left(\frac{\pi (1+\alpha)}{2}\right) \left( \frac{\alpha}{1+\alpha} \int_{-\infty}^{-1} |y|^{1+\alpha} \mu(dy) - \frac{\alpha}{1+\alpha} \int_{1}^\infty y^{1+\alpha} \mu(dy) \right)\Bigg) \\
&= - \frac{\Gamma (1-\alpha)}{-\alpha} \left( \cos \left(\frac{\pi (1+\alpha)}{2}\right) \left(c_1^- + c_1^+\right) - i \sin \left(\frac{\pi (1+\alpha)}{2}\right) \left( c_1^--c_1^+ \right) \right)\\
&= - \frac{\Gamma (1-\alpha)}{-\alpha} \left(c_1^- + c_1^+\right) \cos \left(\frac{\pi (1+\alpha)}{2}\right) \left( 1 - i \frac{ c_1^--c_1^+}{c_1^-+c_1^+} \tan \left(\frac{\pi (1+\alpha)}{2}\right) \right)\\
&= \kappa_{\mathcal{S}_\gamma (\widetilde{\sigma}_{1,\alpha}, \widetilde{\rho}_1, 0)},
\end{align*}
where $\widetilde{\sigma}_{1,\alpha}$ and $\widetilde{\rho}_1$ are given by \eqref{sigma1alpha} and $c_1^-$, $c_1^+$ by \eqref{c-+1}. In the last equality $\sign (\zeta) = 1$ since we suppose $\zeta>0$.

To complete the proof we need to justify taking the limit under the integral in \eqref{e:44}. We denote $g_T(\zeta, x,s)=e^{-x^{-1} \frac{\zeta T}{A_T} (t-s)}$ and split $C \left\{ \zeta \ddagger A_T^{-1} \Delta X^*_{1,2}(Tt) \right\}$ into two parts:
\begin{equation}\label{e:splitI}
C \left\{ \zeta \ddagger A_T^{-1} \Delta X^*_{1,2}(Tt) \right\} = I^{(1)}_{T} + I^{(2)}_{T},
\end{equation}
where
\begin{align}
&\begin{split}\label{e:splitI1}
I^{(1)}_{T} &= \zeta^{1+\alpha} \int_0^\infty \int_{0}^{t} \kappa_{L_!} \left( x \left( 1 - g_T(\zeta, x,s) \right) \right) \frac{\widetilde{\ell}\left( T^{1/(1+\alpha)} \ell^{\#}\left(T \right)^{1/(1+\alpha)} x \zeta^{-1} \right) }{\ell^{\#} \left(T\right)}\\
&\qquad \qquad \qquad \times \alpha x^{-\alpha-2} \1_{[0,1/2]}(g_T(\zeta, x,s)) ds dx,\\
\end{split}\\
&\begin{split}\label{e:splitI2}
I^{(2)}_{T} &= \zeta^{1+\alpha} \int_0^\infty \int_{0}^{t} \kappa_{L_1} \left( x \left( 1 - g_T(\zeta, x,s) \right) \right)  \frac{\widetilde{\ell}\left( T^{1/(1+\alpha)} \ell^{\#}\left(T \right)^{1/(1+\alpha)} x \zeta^{-1} \right) }{\ell^{\#} \left(T\right)} \\
&\qquad \qquad \qquad \times \alpha x^{-\alpha-2} \1_{[1/2,1]}(g_T(\zeta, x,s)) ds dx.
\end{split}
\end{align}
From Potter's bounds \cite[Theorem 1.5.6]{bingham1989regular}, for $0<\delta<\min \left\{ \gamma - 1 - \alpha, \alpha, 1-\alpha\right\}$ there is $C_1$ such that
\begin{equation*}
\frac{\widetilde{\ell}\left( T^{1/(1+\alpha)} \ell^{\#}\left(T \right)^{1/(1+\alpha)} x \zeta^{-1} \right) }{\ell\left( T^{1/(1+\alpha)} \ell^{\#}\left(T \right)^{1/(1+\alpha)}\right)} \leq C_1 \max \left\{ x^{-\delta} \zeta^{\delta}, x^{\delta} \zeta^{-\delta} \right\}.
\end{equation*}
Now from \eqref{e:thmSPL2dBcon} we have that for $T$ large enough
\begin{equation*}
\frac{\widetilde{\ell}\left( T^{1/(1+\alpha)} \ell^{\#}\left(T \right)^{1/(1+\alpha)} x \zeta^{-1} \right) }{\ell^{\#} \left(T\right)} \leq C_2 \max \left\{ x^{-\delta} \zeta^{\delta}, x^{\delta} \zeta^{-\delta} \right\},
\end{equation*}
and hence
\begin{align*}
\left|I^{(1)}_{T}\right| &\leq C_3 \int_0^\infty \int_{0}^{t} \left| \kappa_{L_1} \left( x \left( 1 - g_T(\zeta, x,s) \right) \right) \right| \max \left\{ x^{-\delta}, x^{\delta} \right\}\\
&\hspace{4cm} \times \alpha x^{-\alpha-2} \1_{[0,1/2]}(g_T(\zeta, x,s)) ds dx,\\
\left|I^{(2)}_{T}\right| &\leq C_4 \int_0^\infty \int_{0}^{t} \left| \kappa_{L_1} \left( x \left( 1 - g_T(\zeta, x,s) \right) \right) \right| \max \left\{ x^{-\delta}, x^{\delta} \right\} \\
&\hspace{4cm} \times \alpha x^{-\alpha-2} \1_{[1/2,1]}(g_T(\zeta, x,s)) ds dx.
\end{align*}

We will first show that the dominated convergence theorem may be applied to $I^{(1)}_{T}$ showing that $I^{(1)}_{T}$ converges to the limit in \eqref{e:thm:SLP3}. From \eqref{e:kappaL1alternativeform} by using the inequality 
\begin{equation*}
\left|e^{ix}-\sum_{k=0}^{n} \frac{(ix)^k}{k!} \right| \leq \min \left\{ \frac{|x|^{n+1}}{(n+1)!}, \frac{2 |x|^n}{n!} \right\},
\end{equation*}
we get that for any $x \in \R$,
\begin{equation*}
|\kappa_{L_1}(x)| \leq \int_{\R} \left| e^{ixy}-1-ixy\right| \mu_1(dy) \leq \int_{|xy|\leq 1} |x y|^2 \mu_1(dy) + 2 \int_{|xy|> 1} |x y| \mu_1(dy).
\end{equation*}
Moreover, we have
\begin{equation*}
\sup_{1/2 \leq c \leq 1} \kappa_{L_1}(c x) \leq x^2 \int_{|y|\leq 2/|x|} y^2 \mu_1(dy) +  2 \int_{|xy|> 1} |x y| \mu_1(dy). =: K^{(1)}(x) + K^{(2)}(x),
\end{equation*}
and hence
\begin{equation*}
\left| \kappa_{L_1} \left( x \left( 1 - g_T(\zeta, x,s) \right) \right)  \1_{[0,1/2]}(g_T(\zeta, x,s)) \right| \leq K^{(1)}(x) + K^{(2)}(x).
\end{equation*}
Now
\begin{equation*}
\left|I^{(1)}_{T}\right| \leq C_3 \int_0^\infty \int_{0}^{t}\left( K^{(1)}(x) + K^{(2)}(x) \right) \max \left\{ x^{-\delta}, x^{\delta} \right\}\alpha x^{-\alpha-2} ds dx,
\end{equation*}
and it remains to show this integral is finite. Indeed, we have
\begin{align*}
&\int_0^\infty \int_{0}^{t}  K^{(1)}(x) \max \left\{ x^{-\delta}, x^{\delta} \right\}\alpha x^{-\alpha-2} ds dx\\
&= \alpha t \int_0^1 \int_{|y|\leq 2/|x|} y^2 \mu_1(dy) x^{-\alpha-\delta} dx  + \alpha t \int_1^\infty \int_{|y|\leq 2/|x|} y^2 \mu_1(dy) x^{-\alpha+\delta} dx\\
&= 2^{1-\alpha-\delta} \alpha t \int_0^1 \int_{|y|\leq 1/|x|} y^2 \mu_1(dy) x^{-\alpha-\delta} dx\\
&\hspace{3cm} + 2^{1-\alpha+\delta} \alpha t \int_1^\infty \int_{|y|\leq 1/|x|} y^2 \mu_1(dy) x^{-\alpha+\delta} dx\\
&= 2^{1-\alpha-\delta} \alpha t \int_{|y|\leq 1} y^2 \mu_1(dy) \int_0^1 x^{-\alpha-\delta} dx + 2^{1-\alpha+\delta} \alpha t \int_{|y|>1} y^2 \mu_1(dy) \int_0^{1/|y|} x^{-\alpha-\delta} dx\\
&\hspace{6em} + 2^{1-\alpha-\delta} \alpha t \int_{|y|\leq 1} y^2 \mu_1(dy) \int_1^{1/|y|} x^{-\alpha+\delta} dx\\
&= \frac{2^{1-\alpha+\delta}\alpha t }{1-\alpha-\delta} \int_{|y|>1} |y|^{1+\alpha+\delta} \mu_1(dy) < \infty
\end{align*}
and
\begin{align*}
&\int_0^\infty \int_{0}^{t}  K^{(2)}(x) \max \left\{ x^{-\delta}, x^{\delta} \right\}\alpha x^{-\alpha-2} ds dx\\
&= 2 \alpha t \int_0^1 \int_{|y|> 1/|x|} |y| \mu_1(dy) x^{-\alpha-1-\delta} dx  + 2 \alpha t \int_1^\infty \int_{|y|> 1/|x|} |y| \mu_1(dy) x^{-\alpha-1+\delta} dx\\
&= 2 \alpha t \int_{|y|> 1} |y| \mu_1(dy) \int_{1/|y|}^1 x^{-\alpha-1-\delta} dx  + 2 \alpha t \int_{|y|> 1} |y| \mu_1(dy) \int_1^\infty x^{-\alpha-1+\delta} dx\\
&\hspace{6em} + 2 \alpha t \int_{|y|\leq 1} |y| \mu_1(dy) \int_{1/|y|}^\infty x^{-\alpha-1+\delta} dx\\
&= \frac{2 \alpha t }{-\alpha-\delta}\int_{|y|> 1} |y| \left(1 - |y|^{\alpha+\delta} \right) \mu_1(dy) - \frac{2 \alpha t}{-\alpha-\delta} \int_{|y|> 1} |y| \mu_1(dy)\\
&= \frac{2 \alpha t }{\alpha+\delta}\int_{|y|> 1}|y|^{1+\alpha+\delta} \mu_1(dy)<\infty
\end{align*}
since $1+\alpha+\delta<\gamma$ and $\E |L(1)^{1+\alpha+\delta}|<\infty \Leftrightarrow \int_{|y|>1} |y|^{1+\alpha+\delta} \mu_1(dy) < \infty$.\\

We next show that $I^{(2)}_{T} \to 0$ in \eqref{e:splitI2} as $T\to \infty$. Since $\1_{[1/2,1]}(g_T(\zeta, x,s)) = \1_{\left[ \frac{\zeta  (t-s) T}{A_T \log 2 }, \infty \right)}(x)$, we have by using \eqref{e:LRDcaseboundlinear}
\begin{align*}
&\left| I_{T,2} \right| \leq C_5 \int_0^\infty \int_{0}^{t} x^{-\alpha-1}  \max \left\{ x^{-\delta} , x^{\delta} \right\} \1_{\left[ \frac{\zeta  (t-s) T}{A_T \log 2}, \infty \right)}(x) ds dx\\
&= C_5 \int_0^1 \int_{0}^{t} x^{-\alpha-1-\delta} \1_{\left[ \frac{\zeta u T}{A_T \log 2}, \infty \right)}(x)  dx du  + C_5 \int_1^\infty \int_{0}^{t} x^{-\alpha-1+\delta} \1_{\left[ \frac{\zeta u T}{A_T \log 2}, \infty \right)}(x)  dx du\\
&= C_5 \int_{0}^{t} \1_{\left(0, \frac{A_T \log 2}{\zeta T} \right]}(u) \int_{\frac{\zeta u T}{A_T \log 2}}^1 x^{-\alpha-1-\delta} dx du   + C_5 \int_{0}^{t} \1_{\left(0, \frac{A_T \log 2}{\zeta T} \right]}(u) \int_1^\infty  x^{-\alpha-1+\delta}  dx du\\
&\hspace{6em}  + C_5 \int_{0}^{t} \1_{\left[\frac{A_T  \log 2}{\zeta T}, \infty \right)}(u) \int_{\frac{\zeta u T}{A_T \log 2}}^\infty x^{-\alpha-1+\delta} dx du\\
&= C_6 \int_{0}^{t} \1_{\left(0, \frac{A_T \log 2}{\zeta T} \right]}(u) du - C_7 \left(\frac{T}{A_T}\right)^{-\alpha-\delta} \int_{0}^{t} u^{-\alpha-\delta} \1_{\left(0, \frac{A_T \log 2}{\zeta T} \right]}(u) du\\
&\hspace{6em} + C_8 \left(\frac{T}{A_T}\right)^{-\alpha+\delta} \int_0^t u^{-\alpha+\delta} \1_{\left[\frac{A_T  \log 2}{\zeta T}, \infty \right)}(u) du  \to 0,
\end{align*}
as $T\to \infty$, which completes the proof of \eqref{e:thm:LRD1toprove2}.
\end{proof}

To summarize the results of this subsection, let us assume that \eqref{pifinitemean} (hence \eqref{thm:SRD1case:condition:gamale1} holds) and that $\pi$ has a density $p$ satisfying \eqref{regvarofp} with $\alpha>0$ and some slowly varying function $\ell$. Then the limiting behavior is illustrated in Figure \ref{fig:limits:X1}.

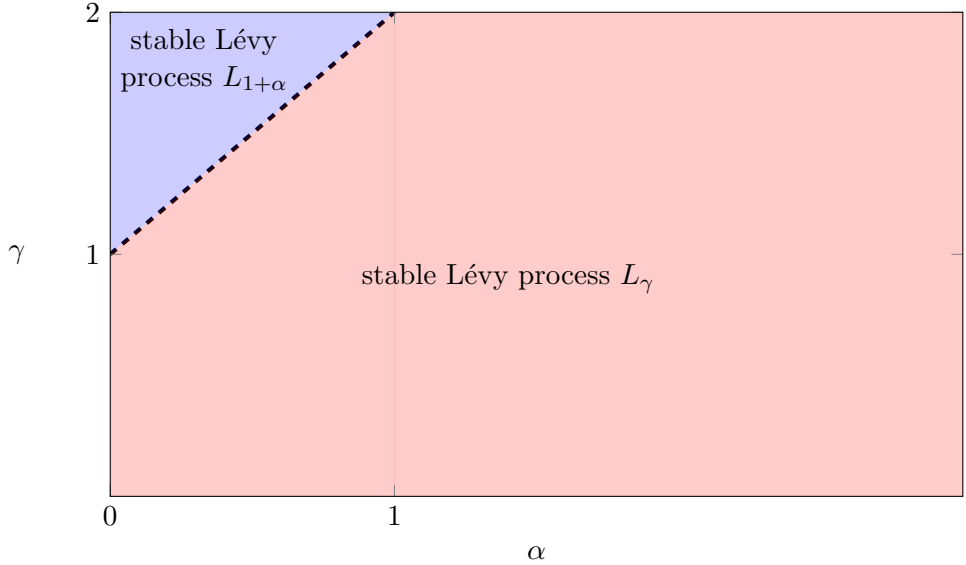
\begin{figure}[h!]
\centering
\begin{tikzpicture}
\begin{axis}[
width=0.8\textwidth,
height=0.5\textwidth,
xmin=0, xmax=3, ymin=0, ymax=2,
xlabel=$\alpha$, 
ylabel=$\gamma$, ylabel style={rotate=-90}, 
xtick={0,1},
xticklabels={$0$,$1$},
ytick={1,2},
axis on top
]
\addplot[dashed, ultra thick] coordinates {(0,1) (1,2)};
\fill[red, opacity=0.2] (axis cs:1,0) -- (axis cs:3,0) -- (axis cs:3,2) -- (axis cs:1,2);
\fill[red, opacity=0.2] (axis cs:0,0) -- (axis cs:1,0) -- (axis cs:1,2) -- (axis cs:0,1);
\node[align=center, text width=0.4\textwidth] at (axis cs:1.4,0.9) {stable L\'evy process $L_{\gamma}$};
\fill[blue, opacity=0.2] (axis cs:0,1) -- (axis cs:1,2) -- (axis cs:0,2);
\node[align=center, text width=0.15\textwidth] at (axis cs:0.33,1.8) {stable L\'evy process $L_{1+\alpha}$};
\end{axis}
\end{tikzpicture}	
\caption{Classification of limits of $X_1^*$}
\label{fig:limits:X1}
\end{figure}

\subsection{The process $X_2^*$}
The background driving L\'evy process of $X_2$ consists only of jumps of magnitude less than or equal to one. The limiting behavior of $X_2^*$ may depend on the growth of the L\'evy measure near the origin. 

Note that $\E |X_2(t)|^q<\infty$ for any $q>0$. In particular, the variance is finite and $\E X_2(t)=0$. Hence, we obtain the following results as a corollary of \cite[Theorems 2.4, 2.2 and 2.3 respectively]{GLT2017Limit}. 

\begin{lemma}\label{lemma:X2SRD}
If 
\begin{equation*}
\int_0^\infty \xi^{-1} \pi(d \xi)< \infty,
\end{equation*}
then as $T\to \infty$
\begin{equation*}
\left\{ \frac{1}{T^{1/2}} X_2^*(Tt) \right\} \overset{fdd}{\to} \left\{\widetilde{\sigma}_2 B(t) \right\},
\end{equation*}
where $\{ B(t)\}$ is standard Brownian motion and
\begin{equation*}
\widetilde{\sigma}_2^2= 2 \sigma_2^2 \int_0^\infty \xi^{-1} \pi(d \xi), \ \text{ with } \ \sigma_2^2=\Var X_2(1)=\frac{1}{2} \int_{|x|\leq 1} x^2 \mu_2(dx).
\end{equation*}
\end{lemma}

\begin{lemma}\label{lemma:X2LRD}
Suppose that $\pi$ has a density $p$ satisfying \eqref{regvarofp} with $\alpha\in(0,1)$ and some slowly varying function $\ell$ and suppose \eqref{LevyMCond} holds with $0\leq \beta<2$. 

\begin{enumerate}[(i)]
\item If 
\begin{equation*}
\beta<1+\alpha,
\end{equation*}
then as $T\to \infty$
\begin{equation*}
\left\{ \frac{1}{T^{1/(1+\alpha)} \ell^{\#}\left(T \right)^{1/(1+\alpha)}} X_2^*(Tt) \right\} \overset{fdd}{\to} \left\{L_{1 + \alpha} (t) \right\},
\end{equation*}
where $\ell^{\#}$ is de Bruijn conjugate of $1/\ell\left(x^{1/(1+\alpha)}\right)$ and  $\{L_{1+\alpha}\}$ is $(1+\alpha)$-stable L\'evy process such that $L_{1+\alpha}(1)\overset{d}{=} \mathcal{S}_{1+\alpha} (\widetilde{\sigma}_{2,\alpha}, \widetilde{\rho}_{2,\alpha}, 0)$ with 
\begin{equation*}
\widetilde{\sigma}_{2,\alpha} = \left( \frac{\Gamma(1-\alpha)}{\alpha} (c^-_2+c^+_2)  \cos \left(\frac{\pi (1+\alpha)}{2}\right) \right)^{1/(1+\alpha)}, \quad \widetilde{\rho}_{2,\alpha} = \frac{c^-_2 - c^+_2}{c^-_2+c^+_2},
\end{equation*}
and  $c^-_2, c^+_2$ given by
\begin{equation*}
c^-_2 = \frac{\alpha}{1+\alpha} \int_{-1}^0 |y|^{1+\alpha} \mu(dy), \qquad c^+_2 = \frac{\alpha}{1+\alpha} \int_0^1 y^{1+\alpha} \mu(dy).
\end{equation*}

\item If 
\begin{equation*}
1+\alpha<\beta<2,
\end{equation*}
then as $T\to \infty$
\begin{equation*}
\left\{ \frac{1}{T^{1-\alpha/\beta} \ell(T)^{1/\beta}} X_2^*(Tt) \right\} \overset{fdd}{\to} \left\{Z_{\alpha, \beta} (t) \right\},
\end{equation*}
where the limit $\{Z_{\alpha, \beta}\}$ is a process defined as in Theorem \ref{thm:mainb=0}(I).
\end{enumerate}
\end{lemma}

Assuming that \eqref{regvarofp} and \eqref{LevyMCond} hold, we can summarize the limiting behavior of $X_2^*$ in Figure \ref{fig:limits:X2}. The value $\alpha=1$ is a boundary between Gaussian and infinite variance stable limit.

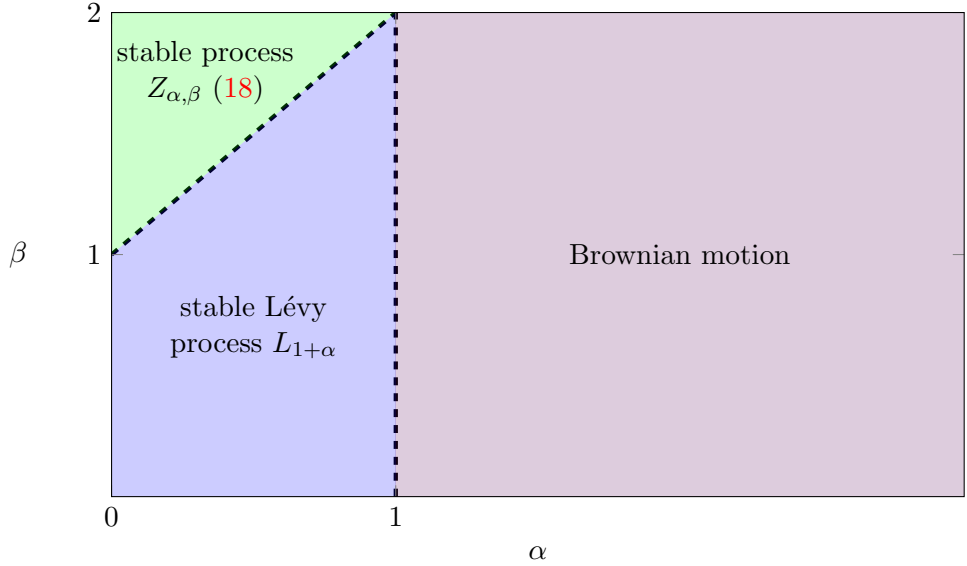
\begin{figure}[h!]
\centering
\begin{tikzpicture}
\begin{axis}[
width=0.8\textwidth,
height=0.5\textwidth,
xmin=0, xmax=3, ymin=0, ymax=2,
xlabel=$\alpha$, 
ylabel=$\beta$, ylabel style={rotate=-90}, 
xtick={0,1},
xticklabels={$0$,$1$},
ytick={1,2},
axis on top
]
\addplot[dashed, ultra thick] coordinates {(1,0) (1,2)};
\addplot[dashed, ultra thick] coordinates {(0,1) (1,2)};
\fill[black!30!violet, opacity=0.2] (axis cs:1,0) -- (axis cs:3,0) -- (axis cs:3,2) -- (axis cs:1,2);
\node[] at (axis cs:2,1) {Brownian motion};
\fill[blue, opacity=0.2] (axis cs:0,0) -- (axis cs:1,0) -- (axis cs:1,2) -- (axis cs:0,1);
\node[align=center, text width=0.15\textwidth] at (axis cs:0.5,0.7) {stable L\'evy process $L_{1+\alpha}$};
\fill[green, opacity=0.2] (axis cs:0,1) -- (axis cs:1,2) -- (axis cs:0,2);
\node[align=center, text width=0.15\textwidth] at (axis cs:0.33,1.75) {stable process $Z_{\alpha,\beta}$ \eqref{Zalphabeta}};
\end{axis}
\end{tikzpicture}	
\caption{Classification of limits of $X_2^*$}
\label{fig:limits:X2}
\end{figure}

\subsection{The process $X_3^*$}
Since $X_3^*$ is a Gaussian process, the limiting behavior is simple (see \cite[Theorems 2.1 and 2.4]{GLT2017Limit}).

\begin{lemma}\label{lemma:gaussiancase}
\begin{enumerate}[(i)]
\item If 
\begin{equation*}
\int_0^\infty \xi^{-1} \pi(d \xi)< \infty,
\end{equation*}
then as $T\to \infty$
\begin{equation*}
\left\{ \frac{1}{T^{1/2}} X_3^*(Tt) \right\} \overset{fdd}{\to} \left\{\widetilde{\sigma}_3 B(t) \right\},
\end{equation*}
where $\{ B(t)\}$ is standard Brownian motion and $\widetilde{\sigma}_3^2= 2 \sigma_3^2 \int_0^\infty \xi^{-1} \pi(d \xi)$ with $\sigma_3^2=\Var X_3(1)=b/2$.

\item Suppose that $\pi$ has a density $p$ satisfying \eqref{regvarofp} with $\alpha\in(0,1)$ and some slowly varying function $\ell$. Then as $T\to \infty$
\begin{equation*}
\left\{ \frac{1}{T^{1-\alpha/2 } \ell(T)^{1/2}} X_3^*(Tt) \right\} \overset{fdd}{\to} \left\{\widetilde{\sigma}_{3,\alpha} B_H(t) \right\},
\end{equation*}
where $\{ B_H(t)\}$ is standard fractional Brownian motion with $H=1-\alpha/2$ and $\widetilde{\sigma}_{3,\alpha} = 2 \sigma_3^2  \frac{\Gamma(1+\alpha)}{(2-\alpha)(1-\alpha)}$ with $\sigma_3^2=\Var X_3(1)=b/2$.
\end{enumerate}
\end{lemma}

\subsection{Proofs of Theorems \ref{thm:mainb=0} and \ref{thm:mainb!=0}}
The limiting behavior of the integrated process $X^*$ follows by combining the limit theorems of the three components in the decomposition \eqref{e:decomposition}. If $X^*$ consists of at least two non-zero components, then each of these may be suitably normalized to obtain a non-trivial limiting process. However, to obtain the limit of the sum of the three components, namely the joint process $X^*$, one has to take the fastest growing among the three normalizations suitable for the components. Hence, the limiting process will depend on the orders of normalizing sequences of the component processes. Namely, an interplay between the parameters $\alpha$, $\beta$ and $\gamma$ will determine the limit.

\begin{proof}[Proof of Theorem \ref{thm:mainb=0}]
The proof is based on comparing the orders of normalizing sequences. Let $E_1$ and $E_2$ denote the exponents of the normalizing sequences for the processes $X_1^*(Tt)$ and $X_2^*(Tt)$ respectively.
\begin{enumerate}[(I)]
\item If $\gamma<1+\alpha$, then $E_1=1/\gamma$ by Lemma \ref{lemma:X1case1}. It is enough to show that $T^{-1/\gamma}X_2^*(Tt)\overset{P}{\to}0$ by showing that $1/\gamma>E_2$. 
\begin{itemize}
\item If $\alpha>1$, then $E_2=1/2$ by Lemma \ref{lemma:X2SRD}. Since $\gamma<2$, $1/\gamma>1/2$.
\item If $\alpha<1$ and $\beta<1+\alpha$, then $E_2=1/(1+\alpha)$ by Lemma \ref{lemma:X2LRD}(i). Since $\gamma<1+\alpha$, we have $1/\gamma>1/(1+\alpha)$.
\item If $\alpha<1$ and $1+\alpha<\beta$, then $E_2=1-\alpha/\beta$ by Lemma \ref{lemma:X2LRD}(ii). We have $1-\alpha/\beta<1+(1-\gamma)/\beta<1+(1-\gamma)/\gamma=1/\gamma$.
\end{itemize}
\item If $1+\alpha<\gamma$, then $E_1=1/(1+\alpha)$ by Lemma \ref{lemma:X1case2}. Note that implicitly we must have $\alpha<1$.
\begin{enumerate}[({II}.a)]
\item If $\beta<1+\alpha$, then $E_2=1/(1+\alpha)$ by Lemma \ref{lemma:X2LRD}(i). We have $E_1=E_2$ and the same normalization, hence the limit is a sum of independent limits obtained in Lemma \ref{lemma:X1case2} and Lemma \ref{lemma:X2LRD}(i). We additionally use \cite[Property 1.2.1]{samorodnitsky1994stable}.
\item If $1+\alpha<\beta$, then $E_2=1-\alpha/\beta$ by Lemma \ref{lemma:X2LRD}(ii). We have $1-\alpha/\beta>1-\alpha/(1+\alpha)=1/(1+\alpha)<$ since $1+\alpha<\beta$.
\end{enumerate}
\end{enumerate}
\end{proof}

\begin{proof}[Proof of Theorem \ref{thm:mainb!=0}]
The proof follows the same arguments as the proof of Theorem \ref{thm:mainb=0}.\\ (I) follows easily from Theorem \ref{thm:mainb=0} and Lemma \ref{lemma:gaussiancase}. For $\alpha>1$ we conclude the statement from the fact that $1/\gamma>1/2$. If $\alpha<1$ and $\gamma<2/(2-\alpha)$, then $\gamma<1+\alpha$, hence we need to compare $1/\gamma$ and $1-\alpha/2$. But this follows easily since $1/\gamma >1-\alpha/2 \Leftrightarrow \gamma < 2/(2-\alpha)$. (II) follows similarly. Indeed, if $2/(2-\alpha) < \gamma < 1+\alpha$, then $1/\gamma <1-\alpha/2$. If $\gamma >1+\alpha$, the rate of growth of the normalizing sequence depends on $\beta$. If $\beta<1+\alpha$, the order of normalizing sequence for $X_1^*(Tt)+X_2^*(Tt)$ is $1/(1+\alpha)$ and $1/(1+\alpha)=1-\alpha/(1+\alpha)<1-\alpha/2$. If $1+\alpha<\beta$, the order of the normalizing sequence for $X_1^*(Tt)+X_2^*(Tt)$ is $1-\alpha/\beta<1-\alpha/2$.
\end{proof}

\textbf{Acknowledgment:} This work was supported by a grant from the Simons Foundation/569118MT at Boston University.

\bigskip
\bigskip

\bibliographystyle{agsm}
\bibliography{References}

\end{document}